\documentclass[peerreview]{IEEEtran}

\usepackage[noadjust]{cite}
\usepackage{amsmath,amssymb,amsfonts,amsthm,bbm}
\usepackage{mathtools}
\usepackage{algorithmicx}
\usepackage{textcomp}
\usepackage{xcolor}

\usepackage[ansinew]{inputenc}          
\usepackage{psfrag}                     
\usepackage[dvips]{graphicx,gincltex,epsfig}            
\usepackage[all]{xy}
\usepackage{a4wide}
\usepackage[dvips]{epsfig}
\usepackage{enumerate}
\usepackage{color}
\usepackage{tikz,tikz-3dplot}
\usepackage{tkz-euclide}
\usetikzlibrary{calc,patterns,through}
\usepackage{nccmath}

\usepackage{array}
\makeatletter
\newcommand\HUGE{\@setfontsize\Huge{34}{60}}
\makeatother   
\usepackage{enumitem, hyperref}
\makeatletter
\def\namedlabel#1#2{\begingroup
    #2%
    \def\@currentlabel{#2}%
    \phantomsection\label{#1}\endgroup
}
\makeatother

\usepackage{algorithm,algpseudocode}

\theoremstyle{plain}
\newtheorem{theorem}{Theorem}

\newtheorem{lemma}[theorem]{Lemma}
\newtheorem{corollary}[theorem]{Corollary}

\theoremstyle{definition}

\newtheorem{example}[theorem]{Example}

\newtheorem{remark}[theorem]{Remark}

\newcommand{\Z}{\mathbb{Z}}
\newcommand{\R}{\mathbb{R}}

\newcommand{\F}{\mathbb{Z}}

\newcommand{\Pa}{\pi}

\newcommand{\Ordo}{\mathcal{O}}

\newcommand{\zero}{\mathbf{0}}
\newcommand{\bu}{\mathbf{u}}
\newcommand{\bv}{\mathbf{v}}
\newcommand{\bw}{\mathbf{w}}
\newcommand{\bd}{\mathbf{d}}
\newcommand{\bs}{\mathbf{s}}
\newcommand{\bc}{\mathbf{c}}

\newcommand{\supp}{\textrm{supp}}
\newcommand{\vsupp}{\textrm{vsupp}}

\newcommand{\bx}{\mathbf{x}}

\newcommand{\ba}{\mathbf{a}}
\newcommand{\bb}{\mathbf{b}}

\newcommand{\by}{\mathbf{y}}
\newcommand{\bp}{\mathbf{p}}

\newcommand{\dmin}{d_{\textrm{min}}}

\newcommand{\cB}{\mathcal{B}}

\newcommand{\cS}{\mathcal{S}}

\newcommand{\ind}{\textrm{ind}}

\title{The Size of the Intersection of $q$-ary Hamming Balls
	\thanks{The authors were funded in part by the Research Council of Finland grants 338797 and 358718. Some of the results in this paper can be found
    in the conference articles of the 2025 IEEE Information Theory Workshop (ITW2025) \cite{pavanITWSubstitution} and
    of the 2026 IEEE International Symposium on Information Theory (ISIT2026) \cite{pavanISITsubstitution}.}
}

\author{%

  \IEEEauthorblockN{Ville Junnila, Tero Laihonen, Tuomo Lehtil{\"a}\\}
  \IEEEauthorblockA{Department of Mathematics and Statistics, 
                    University of Turku, Turku, Finland\\
                    Email: \{viljun, terolai, tualeh\}@utu.fi}\\
                    \smallskip 
     \IEEEauthorblockN{Pavan Padavu Devaraj\\}
  \IEEEauthorblockA{Department of Mathematics and Statistics,
                    University of Turku, Turku, Finland\\
                    CIMFAV \& Instituto de Ingenier\'{i}a Matem\'{a}tica, Universidad de Valpara\'{i}so, Chile\\
                    Email: pavan.padavu@uv.cl}
}
\begin{document}

\maketitle

\begin{abstract}
	   The interest in studying the size of the intersection of multiple $q$-ary Hamming balls has grown due to the recent advances in DNA-based data storage systems. We present an exact formula for the cardinality of the intersection of $s$ Hamming balls of varying radii 
       over a $q$-ary alphabet.
       It is known that the distances between the center points of the Hamming balls are not enough, in general, to determine the size of the intersection. Based on our formula, we are able to find more refined structural properties of the center points for determining the exact size of the intersection.  
       Moreover, we also analyze the size of the intersection of Hamming balls with equal radii $t$ for sufficiently large $n$ when $t$ and  $q$ are constants. When $s=3$, we give the necessary and sufficient conditions (for all $q\ge 2$, $q\neq 6$ and sufficiently large $n$) to obtain the maximum size of the intersection when the centers
       of the Hamming balls have a given minimum distance and demonstrate how to compute it using our general formula. 
\end{abstract}

	
\section{Introduction}

In this paper, we study the intersections of multiple $q$-ary Hamming balls. This problem is important in the Levenshtein sequence reconstruction problem~\cite{Levenshtein}, which is relevant for advanced memory systems such as DNA-storage and associative memories~\cite{yaakobi2018uncertainty}. The problem of intersections of multiple $q$-ary Hamming balls is also related to estimating the number of error-correcting codes~\cite{dong2023number}. 
Although a thorough understanding of such intersections seems rather fundamental concerning the structure of Hamming spaces, rather little is known about them. Most research on the topic has concentrated only on the binary spaces and the cases with at most three balls. In this work, we give a general treatment of the problem that covers cases involving $q$-ary alphabets and multiple balls. We also allow differing radii for the balls.

In what follows, we first agree on some definitions and notation used in the paper. For ease of reference, we gather the main notation of the paper in Table~\ref{tab:notations}. To understand a parameter, we recommend that the reader rely on the appropriate definitions rather than on the table, which only contains concise reminders. Denote the ring of $q$ elements by $\Z_q=\{0,1,\dots, q-1\}$ and the $n$-fold Cartesian product of $\Z_q$ by $\Z_q^n$. For convenience, we represent the set $\{a, a + 1,\dots, b\}$ by $[a, b]$ for integers $a \le b$. The cardinality of a  set $A$ is denoted by $|A|$. The \emph{support} and \emph{value support} of a word $\bx$ are defined as $\supp(\bx) = \{i \mid x_i \neq 0\}$ and $\vsupp(\bx) = \{(i,x_i) \mid x_i \neq 0\}$, respectively. The all-zero word $00\dots0 \in \Z_q^n$ is denoted by $\zero$. For a word $\bx \in \Z^n_q$, the \emph{Hamming weight} $w(\bx)$ is defined by $w(\bx) = |\supp(\bx)|$. Moreover, the \emph{Hamming distance} $d(\bx,\by)$ between two words $\bx,\by \in \Z_q^n$ is defined as the number of coordinate positions in which $\bx$ and $\by$ differ; in other words, $d(\bx,\by) = w(\bx - \by)$ or equivalently, \begin{equation}\label{eqDist}
d(\bx,\by) =w(\bx)+w(\by)-|\supp(\bx)\cap \supp(\by)|-|\vsupp(\bx)\cap \vsupp(\by)|\text.
\end{equation} A \emph{Hamming ball} (resp. \emph{sphere}) of radius $t$ centered at $\bx \in \Z_q^n$ is defined as the set $B_t(\bx) = \{\by \in \Z_q^n \mid d(\bx,\by) \le t\}$ (resp. the set $S_t(\bx) = \{\by \in \Z_q^n \mid d(\bx,\by) = t\}$). For the cardinality of the ball, we use the notation $|B_t(\bx)| = V_q(n,t)$ and clearly, $V_q(n,t)=\sum_{k=0}^{t}\binom{n}{k}(q-1)^k$.  Throughout the paper, we follow the usual convention that $\binom{n}{k} = 0$ if $k < 0$ or $k > n$. Similarly, for integers $n, k_1, \ldots, k_m$ such that $n = k_1 + \cdots + k_m$, we assume the multinomial coefficient
\[
\binom{n}{k_1, k_2, \ldots, k_m} = \binom{n}{k_1} \binom{n-k_1}{k_2} \cdots \binom{n-(k_1+k_2+\cdots+k_{m-1})}{k_m} = 0\text,
\]
if $k_i < 0$ for some $i \in [1,m]$. A nonempty subset $C$ of $\Z_q^n$ is called a \emph{code} and its elements are called \emph{codewords} (usually denoted by $\bc^i$, for $i\in [1,|C|]$). The \emph{minimum distance} $\dmin(C)$ of a code $C$ is defined as $\dmin(C) = \min_{\bc^1, \bc^2 \in C, \bc^1 \neq \bc^2}d(\bc^1, \bc^2)$. A code $C$ is called \emph{$e$-error-correcting} if $\dmin(C) \ge 2e + 1$. 

Let $S = \{\bc^1, \bc^2, \ldots, \bc^s\}$ be a subset of $\Z_q^n$ and $T = (t_1, t_2, \ldots, t_s)$ be a list of nonnegative integers. The intersection of the Hamming balls $B_{t_m}(\bc^m) \ (m \in [1,s])$ is denoted by $I_T(S)$, that is,
\[
I_T(S) = \bigcap_{m=1}^s B_{t_m}(\bc^m) \text.
\]
See Figure~\ref{fig:BallInter} for an illustration of $I_T(S)$. Moreover, we agree that if each radius is equal to $t$ (that is, $t_m=t$ for all  $m \in \{1, \ldots, s\}$), then the notation is simplified as $I_T(S) = I_t(S)$. 
Observe that if two or more of the radii are equal to zero, then the intersection $I_T(S)$ is trivially empty. Moreover, if there is only one ball of radius zero, say $t_i = 0$, then $I_T(S) \subseteq \{\bc^i\}$ and $I_T(S) = \{\bc^i\}$ if and only if $d(\bc^i,\bc^j) \leq t_j$ for all $\bc^j$ with $i \neq j$. Hence, for the rest of the paper, we may assume that $1 \leq t_m \leq n$ for each $m \in [1,s]$.

In this article, we study $|I_T(S)|$.
In~\cite{chll} (and \cite{Iiro}), it has been shown that if $q = 2$ and we have two or three Hamming balls, that is, $|S| \in \{2,3\}$, then the cardinality of $I_T(S)$ is determined solely by the distances between the words of $S$. Moreover, it has been shown in \cite{Iiro} (via counterexamples) that this is not the case if $q=2$ and $|S| \geq 4$ or $q \geq 3$ and $|S| \geq 3$. For illustrative purposes, we give another counterexample for the case $q \geq 3$ and $|S|  = 3$.
\begin{example}\label{exSubInt}
	Assume that $q \geq 3$. Let $S=\{\bc^1, \bc^2, \bc^3\}\subset \Z_q^4$ where $\bc^1=0000,\bc^2=1100$ and $\bc^3=0110$. The pairwise distance between each of these words is $2$. The words in $I_2(S)$ can be presented as $0j00,i100,01i0,010i,$ and $1010$ where $j\in[0,q-1]$ and $i\in[1,q-1]$ for $q\geq3$. Hence, we have $|I_2(S)|=4q-2$. Consider next $S'=\{\bw^1,\bw^2,\bw^3\}\subset \Z_q^4$ for $\bw^1=0000,\bw^2=1100$ and $\bw^3=2200$.  The pairwise distance between each of these words is also $2$. The words in $I_2(S')$ can be presented as $jj'00$ for $j,j'\in[0,q-1]$.  Hence, we have $ |I_2(S')|=q^2$. Since $|I_2(S)|\neq |I_2(S')|$ for all integers $q\geq 3$, we can see that the size of intersections is \emph{not} determined by the distances only.  
\end{example}

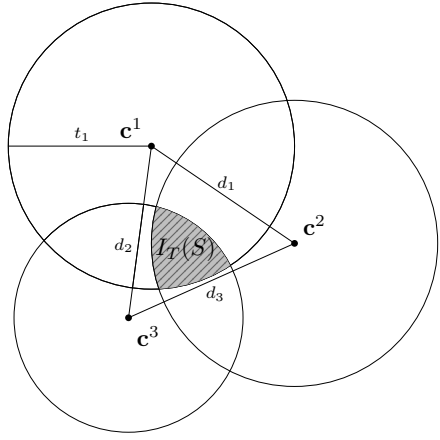
\begin{figure}\centering
\scalebox{0.9}{
\begin{tikzpicture}[scale=0.84][thick]
\draw (0,0.5) circle (2.5cm);
    \coordinate[label = above left:] (X) at (-2.5,0.5);
    \coordinate[label = above left:] (O) at (0,0.5);

\draw (2.5,-1.2) circle (2.5cm);

\draw (-0.4,-2.5) circle (2cm);

    \coordinate[label = above left:$\mathbf{c}^1$] (A) at (0,0.5);
    \coordinate[label = above right:$\mathbf{c}^2$] (D) at (2.5,-1.2);
    \coordinate[label = below right:$\mathbf{c}^3$] (E) at (-0.4,-2.5);
    \draw (A) -- (D) -- (E) -- (A);
    \draw (X) -- (O);

\coordinate[label = above right:${\scriptstyle d_1}$] (d1) at (1,-0.42);
\coordinate[label = above right:${\scriptstyle d_2}$] (d2) at (-0.8,-1.5);
\coordinate[label = above right:${\scriptstyle d_3}$] (d3) at (0.8,-2.35);
\coordinate[label = above right:${\scriptstyle t_1}$] (d3) at (-1.5,0.45);

        \node at (A)[circle,fill,inner sep=1pt]{};
        \node at (D)[circle,fill,inner sep=1pt]{};
        \node at (E)[circle,fill,inner sep=1pt]{};

\draw [clip](0,0.5) circle (2.5cm);
\draw [clip] (-0.4,-2.5) circle (2cm);
\fill[lightgray] (2.5,-1.2) circle (2.5cm);
\fill[pattern=north east lines, draw, pattern color=darkgray!70!lightgray] (2.5,-1.2) circle (2.5cm);
\coordinate[label = above right:$I_T^d(S)$] (It) at (0.8,1.35);
    \node at (0.6,-1.3)  {$I_T(S)$
};

    \draw (A) -- (D) -- (E) -- (A);
    \draw (X) -- (O);

\end{tikzpicture}}
    \caption{The intersection $I_T(S)$ of three balls centered at words $\bc^1$, $\bc^2$ and $\bc^3$ with pairwise distances $d_1$, $d_2$ and $d_3$, where $T = (t_1, t_2, t_3)$.} 
    \label{fig:BallInter}
\end{figure}

Notice that in the example we have $|\supp(\bc^2) \cap \supp(\bc^3)| = 1 \neq 2 = |\supp(\bw^2) \cap \supp(\bw^3)|$ (although the distances between the centers of the balls are the same in both cases). In Section~\ref{subsecNumPara}, it is shown that this difference is crucial to the size of the intersection and that the exact size can actually be determined based on the distances of the centers and $|\supp(\bc^2) \cap \supp(\bc^3)|$. In the section, the required additional parameters (besides the distances) to decide the size of the intersection are described in general for any number of centers.

\subsection{Related Work} \label{SecRelatedWork}
Recall (from~\cite{chll}) that if $q = 2$ and $|S| \in \{2,3\}$, then the cardinality of $I_T(S)$ is based solely on the distances between the words of $S$. However, in~\cite{chll}, the exact values of the cardinalities were left open. Later, in~\cite[Corollary~1]{Levenshtein}, the exact cardinality of the intersection $I_t(S)$ was determined when 
$|S| = 2$ and $q \geq 2$. Moreover, some progress on the size of $I_t(S)$ has been made concerning the maximum cardinality $N_t^q(s,d)$ of $I_t(S)$ with respect to $|S|$ and $\dmin(S)$ which is defined as
\[
N_t^q(s,d) = \max\left\{ |I_t(S)| \mid S \subseteq \Z_q^n, |S| = s, \dmin(S) \geq d \right\}\text.
\]
This can be seen through the following existing results:
\begin{itemize}
	\item In~\cite[Theorem~6]{yaakobi2018uncertainty}, $N_t^q(s,d)$ has been determined for $s = 3$ and $q = 2$ (when $n$ is large enough).
    \item According to~\cite[Theorem 33]{junnila2020levenshtein}, if $q=2$, and $e$, $\ell$  and $a\in[0,\ell-1]$ are constants as well as $V_2(n,\ell-a-1)+1\leq N_{e+\ell}^2(s,2e+1)\leq V_2(n,\ell-a)$, then $s\in\Theta(n^a)$.
	\item In~\cite[Theorem~12]{yaakobi2018uncertainty}, it has been shown that for $q=2$ and constants $t$, $s$ and $d$, we have $N_t^2(s,d)\in \Theta(n^{t-\lceil d/2\rceil})$.
	\item If we have sufficiently large $n$ and constants $e$, $\ell$, $s$ and $q=2$,  then we have $N_{e+\ell}^2(\ell+1,2e+1)= V_2(n,\ell-1)+1$ and $N_{e+\ell}^2(s,2e+1)=V_2(n,\ell-1)$ for each $s\geq\ell+2$ (see  \cite{junnila2020levenshtein}). Furthermore, similar results in the case with $q \geq 3$ have been discussed in~\cite{junnila2025alphabet}. 
    \item If $n\geq \ell$, $q=2$ and $N^2_{e+\ell}(s,2e+1)\geq V_2(n,\ell-1)+1$, then $s\leq 2^\ell$ (see  \cite[Theorem~7]{junnila2020levenshtein}).
	\item In~\cite[Theorems~11 and 12]{junnila2024levenshtein}, given sufficiently large $n$, an exact formula for $N_{t}^2(s,d)$ has been given in the binary case for some constants $s$, $e$ and $\ell$ 
    where $t=e+\ell$, $d=2e+1$ and $s\leq \ell+1$.
\end{itemize}

However, it should be noted that in the results 
mentioned above regarding $N_{t}^q(s,d)$, the problem of intersecting balls is considered from a somewhat different angle. Namely, their formulation is based on the Levenshtein sequence reconstruction problem, for which the value $N_t^q(s,d)$ plays an important role. 
In what follows, we briefly discuss the Levenshtein sequence reconstruction problem and its application to DNA-storage.

Let $S \subseteq \Z_q^n$ be a code. In the Levenshtein sequence reconstruction model, a word $\bx\in S$ is transmitted through $N$ channels and, due to the errors introduced in each channel, a set of output words $Y = \{\by_1, \by_2, \dots, \by_N\}$ is obtained. The output of each channel is assumed to be distinct and the number of errors occurring in a channel is at most $t$. Let $T(Y) \subseteq S$ denote the set of codewords that may produce the output set $Y$ when sent through the $N$ channels. In other words, $T(Y)$ can be considered as a list of possible transmitted codewords which the decoder can deduce based on the output words $Y$. The maximum value $N_t^q(s,d)$ is essential in the reconstruction model since if $|Y| \geq  N_t^q(s,d) + 1$, then the length of the list $T(Y)$ is bounded from above by $s-1$. On the other hand, if $|Y| \leq  N_t^q(s,d)$, then it is possible that $|T(Y)| \geq s$.

The original motivation for the Levenshtein sequence reconstruction model came from molecular biology and chemistry~\cite{Levenshtein}. 
Recently, the topic has regained popularity due to its connections to advanced memory storage technologies such as associative memories \cite{yaakobi2018uncertainty}, racetrack memories \cite{chee2018reconstruction} and, especially, polymer-based memories such as ones based on DNA \cite{Uusi_Maria_Abu-Sini}. 
The relationship between DNA-based memory storage systems and the Levenshtein reconstruction model can be stated as follows: In such memories, the data is stored directly in the DNA by synthesizing suitable strands. This leads to multiple redundant copies of each strand, which are stored in a pool of freely floating DNA. When the information is read, we pick strands from this pool (possibly obtaining multiple erroneous copies of the strand in which the original information was stored). These strands are multiplied in the reading process (called sequencing). Finally, we end up with many copies of a strand, each with its own errors. We may apply the Levenshtein reconstruction model to this problem by considering each strand as an output word in $Y$ based on which we can determine the originally stored information $\bx$ (see \cite{heckel2019characterization, Abu-Sini2024}).
The larger alphabets (with $q \geq 3$) studied in this paper are relevant in the polymer-based technologies~\cite{laurent2020high,MR3376124}; in particular, the four nucleotide bases of the DNA emphasize the importance of the case with $q = 4$ symbols.

\smallskip

There is also a connection between $N_t^q(s,d)$ (or $I_t(C)$) and the error-correcting codes. If $C$ is an $e$-error-correcting code with $d=2e+1$, then $N_t^q(s,d)$ gives the maximum intersection of $s$ distinct $t$-radius balls centered at the codewords~\cite{yaakobi2018uncertainty}. Furthermore, the value is also related to list-error-correcting codes~\cite{Guruswamin_kirja}:
If $N_t^q(s+1,d)=0$ and $\dmin(C)\geq d$ for $C\subseteq\Z^n_q$, then $|B_t(\bw)\cap C|\leq s$ for any $\bw\in \Z^n_q$ (see Remark~\ref{RemList} for more discussion on this topic).

\smallskip
Beyond the study of intersections of multiple Hamming balls, several works have focused specifically on the intersection of two balls, motivated by a range of applications.
In~\cite[Theorem 14.2.2]{blahut1983theory}, Blahut considers the intersection number $N(l,h;d)$ of the Hamming scheme and also gives an exact formula for it. Using the notation of our paper, the intersection number can be defined as $N(l,h;d) = |S_{l}(\bc^1) \cap S_{h}(\bc^2)|$, where $\bc^1$ and $\bc^2$ are words of $\Z_q^n$ such that $d(\bc^1, \bc^2) = d$. 
Therefore, the cardinality of $B_{t_1}(\bc^1) \cap B_{t_2}(\bc^2)$, where $d(\bc^1, \bc^2) = d$, can be expressed using the intersection numbers as $|B_{t_1}(\bc^1) \cap B_{t_2}(\bc^2)| = \sum_{l=0}^{t_1}\sum_{h=0}^{t_2} N(l,h;d)$.
The motivation for the study of intersection numbers was their connections to the probabilities of decoding errors and failures. Later in~\cite{barg2000strengthening}, these values have been used to improve the Gilbert-Varshamov bound. Moreover, \cite{dong2023number, balogh2016applications, kim2022exponential} have given approximations for the value $|I_t(\{\bx,\by\})|$. The works~\cite{dong2023number, balogh2016applications} were motivated by estimating the number of $e$-error correcting codes while~\cite{kim2022exponential} cites multiple motivations such as improvements of the Gilbert-Varshamov bound and list-error-correcting codes. Furthermore, the quantity $|I_t(\{\bx,\by\})|$ has come up also in the case of analyzing immunological data in~\cite{smith1997deriving}. 
%
%
In~\cite{alon2024helly}, the intersection of Hamming balls has been considered from yet another perspective, namely through the lens of Helly numbers (originally considered in~\cite{helly1923mengen}). 

\subsection{Structure of the Paper}
 In this paper, we focus on the cardinality of the intersection $I_T(S)$. First, in Section~\ref{secGeneralFormula}, we give a general formula for computing the size of the intersection $I_T(S)$ and analyze this formula. We continue in Subsection~\ref{subsecNumPara} by studying the properties of the center points which determine the size of the intersection. As we will see, the distance between the center points is not enough to determine the size of $I_T(S)$, when $|S|\ge 4$ and $q=2$ or $|S|\ge 3$ and $q\ge 3$. Then in Section~\ref{SecIntersection}, we focus on the asymptotics (in $n$) of the case $I_t(S)$, that is, the case where each ball has the same radius. We begin by giving some general lemmas and continue in Subsection~\ref{sec3balls} by studying the case of the three balls for $q\ge 3$. We give the triples obtaining the maximum size of the intersection together with the formula for the intersection size $N_q^t(3,d)$. Finally, we conclude in Section~\ref{secConclusion}.

\begin{table}[h]
    \centering
    \caption{Terminology and Notation}
    \renewcommand{\arraystretch}{1.4}
    \begin{tabular}{cc}
        \hline
        Notation & Description \\
        \hline
        &\textbf{Notation Used Throughout the 
        Paper}\\
        $[a, b]$ & The set $\{a, a + 1,\dots, b\}$ for integers $a \le b$.\\
        $B_t(\bx)$ &  The Hamming ball $\{\by \in \Z_q^n \mid d(\bx,\by) \le t\}$.\\
        $S_t(\bx)$& The Hamming sphere $\{\by \in \Z_q^n \mid d(\bx,\by) = t\}$.\\
        $\supp(\bx)$& Non-zero coordinates of $\bx$, i.e., $ \{i \mid x_i \neq 0\}$.\\ 
        $\vsupp(\bx)$ & Non-zero coordinates and their values, i.e., $ \{(i,x_i) \mid x_i \neq 0\}$.\\
        $I_T(S)$ & The intersection $\bigcap_{m=1}^s B_{t_m}(\bc^m)$, $t_m\in T$. \\
        $N_t^q(s,d)$ & $\max\left\{ |I_t(S)| \mid S \subseteq \Z_q^n, |S| = s, \dmin(S) \geq d \right\}$.
        \\ \hline
        &\textbf{Notation Used in Section~\ref{SecFormula}}\\
        $\cS(s,k)$ & The Stirling number of the second kind.\\
        $P_{k,i}$ & 
        The $i$th partition of $[1,s]$ into $k$ blocks.\\
        $\Pa_{k,i}$ & Coordinate positions corresponding to $P_{k,i}$.\\
        $p_{k,i}$ & The number of coordinate positions in $\Pa_{k,i}$.\\
        $P_{q}(s)$ & The total number of partitions $P_{k,i}$, which is equal to $ \sum_{k=1}^{q} \cS(s,k)$.
        \\
        $f_{k,i}(\bc^j)$ & The function indicating the index of the block to which $j$ belongs in $P_{k,i}$.\\
        $j_{k,i}^l$ & The number of coordinates of $\bu \in I_T(S)$ equal to the symbol $l-1$ within $\Pa_{k,i}$.\\
        $\mathcal{J}$ & 
        All the tuples of integers $j_{k,i}^l$ satisfying the conditions of Theorem~\ref{Thm_intersection_size}. \\
        \hline
        &\textbf{Notation Used in Section~\ref{SecIntersection}}\\
        $d_M(\bv,S)$ & The smallest $r$ such that $S \subseteq B_r(\bv)$, i.e.,  $\max\{d(\bv,\bs)\mid \bs\in S\}$.\\

$C_d(S)$ & The words whose maximum distance w.r.t. $S$ equals $d$, i.e., $\{\bv\in \Z^n_q\mid d_M(\bv,S)=d\}$.\\ 
$d(S)$ & The radius of the set $S$, i.e., $\min\{d_M(\bx, S) \mid \bx \in \F^n_q\}$.\\  
$C(S)$ & Central words $\{\bv\in \Z_q^n\mid d_M(\bv,S)= d(S)\}$.\\
$A$ & The cardinality of $\supp(\bc^1)\cap \supp(\bc^2)$.\\
$S_d^A$ &  An equidistant triple $\{\bc^0,\bc^1,\bc^2\}$ where $|\supp(\bc^1)\cap \supp(\bc^2)|=A$, when $\bc^0=\zero$.\\
$A_1$ & The number of coordinates in $\supp(\bc^1)\cap \supp(\bc^2)$ in which $\bc^1$ and $\bc^2$ agree.\\
$A_2$ & The number of coordinates in $\supp(\bc^1)\cap \supp(\bc^2)$ in which $\bc^1$ and $\bc^2$ disagree.\\
$\bw_{|n'}$ & The restriction of $\bw\in \Z^n_q$ to $\Z^{n'}_q$ for $n'\le n$.\\

\hline
        
    \end{tabular}
    
    \label{tab:notations}
\end{table}

\section{The Exact Cardinality of the Intersection}\label{SecFormula} 

In the first part of this section, we present an exact formula for the size of the intersection of $s$ substitution balls in the $q$-ary Hamming space $\Z^n_q$, where the radius $t_m$ ($m \in [1,s]$) of each ball can be chosen individually (see Theorem~\ref{Thm_intersection_size}). As a by-product, for $|S| = 3$ and $q = 2$, we extend a previously known result for equal radii to the case with varying radii; see  
Corollary~\ref{Cor3balls} and Remark~\ref{Rem3balls}.
As discussed in Example~\ref{exSubInt}, the pairwise distances are usually not enough to determine the size of the intersection. In the second part of this section, we consider based on the obtained exact formula, the additional information (besides the pairwise distances) we need about the words in $S$ to calculate $|I_T(S)|$. Finally, we conclude the section with Remark~\ref{RmkComplexity} by showing that the exact formula of Theorem~\ref{Thm_intersection_size} can be efficiently computed when the number $s$ of the balls is constant (as is natural to assume for DNA-storages).

\subsection{The General Formula for the Cardinality of the Intersection}\label{secGeneralFormula}

Let $s,q \geq 2$ and $S = \{\bc^1, \bc^2, \dots, \bc^s\} \subseteq \Z_q^n$ consist of the $s$ distinct centers of the balls. In order to determine the size of the intersection, we consider for each coordinate position $h \in [1,n]$ the $s$-tuple $(c^1_h, c^2_h, \dots, c^s_h)$ formed by the $h$th coordinates of the words of $S$. Then we form a \emph{partition} $P_h$ of $[1,s]$ based on the tuple as follows: $j$ and $j'$ of $[1,s]$ belong to the same \emph{block} (or \emph{equivalence class}) of $P_h$ if and only if $c^j_h = c^{j'}_h$, i.e., the blocks are composed of the indices $m$ of the words $\bc^m$ with equal values at the $h$th coordinate. 
Then we group together the coordinate positions $h$ for which the corresponding partitions $P_h$ are equal and show in Theorem~\ref{Thm_intersection_size} that the numbers of coordinate positions in such groups are essential in determining the size of the intersection. Let us now discuss this general idea in more detail.

The partitions $P_h$ may be regarded as divisions of $[1,s]$ into nonempty subsets. Such partitions have been extensively studied within the field of combinatorics~\cite{DezaSN}; in particular, they are closely related to the Bell numbers and the Stirling numbers of the second kind. The \emph{Stirling number of the second kind}  gives the number of ways to partition $[1,s]$ into $k$ nonempty subsets and is denoted by $\cS(s,k)$. In~\cite[Section~2.2~(p.~161)]{DezaSN}, a closed formula for $\cS(s,k)$ is given:
\[
\cS(s,k) = \sum_{i=1}^k \frac{(-1)^{k-i}i^s}{(k-i)!i!} \text.
\]
Furthermore, the \emph{Bell number} $\cB_s$ counts all the possible partitions of $[1,s]$ and it can be expressed using $\cS(s,k)$ as follows \cite[Section~2.2~(p.~84)]{DezaSN}:
\[
\cB_s = \sum_{k=1}^s \cS(s,k) \text.
\]
Recall that we are interested in the partitions $P_h$ of $[1,s]$ based on the $s$-tuple $(c^1_h, c^2_h, \dots, c^s_h)$ such that $j$ and $j'$ of $[1,s]$ belong to the same block if and only if $c^j_h = c^{j'}_h$. Observe that in addition to $s$, the number of blocks in such a partition is upper bounded by $q$ (since there are at most $q$ choices for each $c^j_i$). Thus, the maximum number $P_q(s)$ of such partitions $P_h$ can be given as follows: 
\begin{equation} \label{Eq_Pqs}
	P_{q}(s) = \sum_{k=1}^{\min\{q,s\}} \cS(s,k) = \sum_{k=1}^{q} \cS(s,k)\text.
\end{equation}
The latter equality follows due to the fact that $\cS(s,k) = 0$ if $k > s$; see, for example, \cite[Remark 3.1]{bagui2024stirling}. Moreover, if $q \geq s$, then $P_{q}(s)$ is equal to the Bell number $\cB_s$. For future purposes, we use the following notation for the $P_{q}(s)$ distinct partitions of $[1,s]$: $P_{k,i}$ denotes a partition of $[1,s]$ into $k$ blocks, where $k \in \{1, 2, \dots, \min\{q,s\}\}$ and $i \in \{1, 2, \dots, \cS(s,k)\}$. Observe that each partition $P_h$ is equal to some $P_{k,i}$; however, there might be partitions $P_{k,i}$ which are not equal to any $P_h$. Hence, the number of distinct partitions $P_h$ is at most $P_{q}(s)$.

With respect to the discussion of the first paragraph of the section, we group together the coordinate positions $h \in [1,n]$ which correspond to the same partition $P_{k,i}$ and denote the group of such coordinate positions by $\Pa_{k,i} \subseteq [1,n]$. Moreover, if $P_{k,i}$ is not equal to any $P_h$, then we agree that $\Pa_{k,i} = \emptyset$. Observe that the subsets $\Pa_{k,i}$ form a sort of a partition of $[1,n]$, where some $\Pa_{k,i}$ are allowed to be empty; in what follows, such a partition is called a \emph{pseudo-partition}\footnote{Notice that the term 'pseudo-partition' also has some other meanings in the literature; for example, see~\cite{BonacinaGalesi2013}.}. Moreover, the number of coordinate positions in $\Pa_{k,i}$ is denoted by $p_{k,i}$, i.e., $|\Pa_{k,i}| = p_{k,i}$. Recall that for some $k$ and $i$, the set $\Pa_{k,i}$ of coordinate positions might be empty, i.e., $p_{k,i} = 0$, and observe that the values $p_{k,i}$ sum up to $n$, when $k \in \{1, 2, \dots, \min\{q,s\}\}$ and $i \in \{1, 2, \dots, \cS(s,k)\}$ go through all the possible choices. Furthermore, assume that the order of the $k$ blocks of each $P_{k,i}$ is fixed as follows (although subsets of a partition are usually assumed to be unordered): the blocks are listed in an ascending order based on their smallest element. Let $\ind: S \to \{1,2,\dots,s\}$ be a function such that $\ind(\bc) = j$ if $\bc = \bc^j$, that is, the function $\ind$ outputs the index of $\bc$ among set $S$; the function $\ind$ is well-defined since the words of $S$ are distinct. Now we may define a function $f_{k,i}: S \to \{1, \dots, k\}$ such that $f_{k,i}(\bc) = \ell$ if the index $\ind(\bc) = j$ belongs to the $\ell$th subset of $P_{k,i}$; indeed, $f_{k,i}$ is well-defined since the order of the blocks of $P_{k,i}$ is fixed. The introduced notations are illustrated in the following example.
\begin{example} \label{Ex_notation}
	Let $q \geq 3$ and consider an intersection of three balls. By~\eqref{Eq_Pqs}, we now have $P_q(3) = \cS(3,1) + \cS(3,2) + \cS(3,3) = 1 + 3 + 1 = 5$; in particular, the subsets $\Pa_{1,1}$, $\Pa_{2,1}$, $\Pa_{2,2}$, $\Pa_{2,3}$ and $\Pa_{3,1}$ form a pseudo-partition of $[1,n]$ with the corresponding partitions of $[1,s] = [1,3]$ being $P_{1,1} = (\{1,2,3\})$, $P_{2,1} = (\{1,2\}, \{3\})$, $P_{2,2} =  (\{1,3\}, \{2\})$, $P_{2,3} = (\{1\}, \{2,3\})$ and $P_{3,1} = (\{1\},\{2\},\{3\})$ (where the list notation emphasizes that the blocks are listed in the fixed order declared above). Let $S = \{\bc^1, \bc^2, \bc^3\}$ be a set, where $\bc^1 = 0000$, $\bc^2 = 1100$ and $\bc^3 = 0110$ (see Example~\ref{exSubInt}). Now we have $P_1 = \{\{1,3\},\{2\}\}$, $P_2 = \{\{1\},\{2,3\}\}$, $P_3 = \{\{1,2\},\{3\}\}$ and $P_4 = \{\{1,2,3\}\}$. Therefore, we obtain $\Pa_{1,1} = \{4\}$, $\Pa_{2,1} = \{3\}$, $\Pa_{2,2} = \{1\}$, $\Pa_{2,3} = \{2\}$ and $\Pa_{3,1} = \emptyset$, which imply $p_{1,1} = p_{2,1} = p_{2,2} = p_{2,3} = 1$ and $p_{3,1} = 0$. Furthermore, illustrating the functions $f_{k,i}(\bc)$, we have, for example, $f_{1,1}(1100) = 1$ and $f_{2,1}(0110) = 2$ since $\ind(1100) = 2$ and $\ind(0110) = 3$ as well as $2$ belongs to the first block of $P_{1,1}$ and $3$ to the second block of $P_{2,1}$.
\end{example}

In the following theorem, we give an exact formula for the number of words in the intersection of balls of radii $t_1, t_2, \ldots, t_s$ respectively centered at $\bc^1, \bc^2, \dots, \bc^s$ based on the values $p_{k,i}$.
\begin{theorem} \label{Thm_intersection_size}
	Let $S = \{\bc^1, \bc^2, \dots, \bc^s\}$ be a subset of $\Z^n_q$ and $T = (t_1, t_2, \ldots, t_s)$ a list of positive integers, where $s\geq2$. Let $j_{k,i}^l$ be nonnegative integers, where $k \in \{1, 2, \dots, \min\{q,s\}\}$, $i \in \{1, 2, \dots, \cS(s,k)\}$ and $l \in \{1, 2, \dots, \min\{q,k+1\}\}$, satisfying the following conditions: (i) $\sum_{l=1}^{\min\{q,k+1\}} j_{k,i}^l = p_{k,i}$ and (ii) for each $\bc^m \in S$,
	\[
	\sum_{k=1}^{\min\{q,s\}} \sum_{i=1}^{\cS(s,k)} \left( p_{k,i} - j_{k,i}^{f_{k,i}(\bc^m)} \right) \leq t_m \text.
	\]
    Let $\mathcal{J}$ denote all the tuples formed by $j_{k,i}^l$ which satisfy the conditions~(i)~and~(ii). The formula for the size of the intersection divides into the following two cases depending on whether  (a) $q > s$ or (b) $q \leq s$:
	\begin{itemize}

		\item[(a)] If $q > s$, then
		\begin{equation} \label{Eq2ndFormula}
		    |I_T(S)| = \sum_{\mathcal{J}} \left( \prod_{k=1}^{s} \prod_{i=1}^{\cS(s,k)} \binom{p_{k,i}}{j_{k,i}^1, j_{k,i}^2, \dots, j_{k,i}^{k+1}}(q-k)^{j_{k,i}^{k+1}} \right)  \text,
		\end{equation}
		where the sum goes through all the tuples of $\mathcal{J}$.

        \item[(b)] If $q \leq s$, then
		\begin{equation} \label{Eq1stFormula}
		    |I_T(S)| = \sum_{\mathcal{J}} \left( \prod_{k=1}^{q-1} \left(  \prod_{i=1}^{\cS(s,k)} \binom{p_{k,i}}{j_{k,i}^1, j_{k,i}^2, \dots, j_{k,i}^{k+1}}(q-k)^{j_{k,i}^{k+1}} \right)  \cdot \prod_{i=1}^{\cS(s,q)} \binom{p_{q,i}}{j_{q,i}^1, j_{q,i}^2, \dots, j_{q,i}^{q}} \right)  \text,
		\end{equation}
	    where the sum goes through all the tuples of $\mathcal{J}$.
	\end{itemize}
\end{theorem}
\begin{proof}
	Since the Hamming distance counts coordinatewise the differences of words (regardless of the actual different symbols in the words), we may assume, without loss of generality, that the words of $S$ are such that within $\Pa_{k,i}$ their symbols in the coordinates belong to $[0,k-1] \subseteq \Z_q$, i.e., the $k$ blocks of the partition $P_{k,i}$ are formed by the $k$ smallest elements of $\Z_q$. More formally, we may assume that a word $\bc^m \in S$ is such that for each $k \in \{1, 2, \dots, \min\{q,s\}\}$ and $i \in \{1, 2, \dots, \cS(s,k)\}$, its value in the coordinates of $\Pa_{k,i}$ is equal to the element of $\Z_q$ corresponding to $f_{k,i}(\bc^j) - 1$; in particular, the symbols of $\bc^m \in S$ within $\Pa_{k,i}$ belong to $[0,k-1] \subseteq \Z_q$. Thus, the word $\bc^m \in 
    [0, \min\{q,s\}-1]^n$. Let $\bu$ be a word of $\Z_q^n$. For $k \in \{1, 2, \dots, \min\{q,s\}\}$, $i \in \{1, 2, \dots, \cS(s,k)\}$ and $l \in \{1, 2, \dots, \min\{q,k+1\}\}$, the nonnegative integers $j_{k,i}^l$ may be considered as the distribution of symbols of $\bu$ among the coordinate positions of $\Pa_{k,i}$ as follows:
	\begin{itemize}
		\item If $q > k$, then for each $l \in [1,k]$, the value $j_{k,i}^l$ corresponds to the number of coordinates of $\bu$ equal to $l-1$ within $\Pa_{k,i}$ and $j_{k,i}^{k+1}$ to the number of coordinates belonging to $[k,q-1]$.
		\item If $q \leq k$, then for each $l \in [1,q]$, the value $j_{k,i}^l$ corresponds to the number of coordinates of $\bu$ equal to $l-1$ within $\Pa_{k,i}$.
	\end{itemize}
    Thus, as $j_{k,i}^l$ represents the distribution of symbols of $\bu$ in the coordinates indicated by $\Pa_{k,i}$, the condition~(i) $\sum_{l=1}^{\min\{q,k+1\}} j_{k,i}^l = p_{k,i}$ is satisfied. Furthermore, for each $\bc^m$, we have
	\[
	d(\bu, \bc^m) = \sum_{k=1}^{\min\{q,s\}} \sum_{i=1}^{\cS(s,k)} \left( p_{k,i} - j_{k,i}^{f_{k,i}(\bc^m)} \right) \text.
	\]
	Indeed, the previous equality holds since $p_{k,i} - j_{k,i}^{f_{k,i}(\bc^m)}$ coordinate positions contribute to the distance $d(\bu, \bc^m)$ among $\Pa_{k,i}$. In what follows, we show that $\bu \in I_T(S)$ if and only if the conditions~(i) and (ii) are satisfied.
    \begin{itemize}
        \item[($\Rightarrow$)] If $\bu \in I_T(S)$, then by the previous discussion the conditions~(i) $\sum_{l=1}^{\min\{q,k+1\}} j_{k,i}^l = p_{k,i}$ and (ii) $d(\bu, \bc^m) \leq t_m$ for all $\bc^m \in S$ are satisfied for all $\bc^m \in S$.
        \item[($\Leftarrow$)] Assume that the word $\bu$ satisfies the conditions~(i) and (ii). Firstly, by~(i), it follows that the values $j_{k,i}^l$ give a proper distribution of symbols of $\Z_q$ in the coordinates positions of $\Pa_{k,i}$. Secondly, the condition~(ii) guarantees that $d(\bu, \bc^m) \leq t_m$ for all $m \in \{1, \ldots, s\}$, i.e., $\bu \in I_T(S)$.
    \end{itemize}
    Hence, in order to prove the claim, we count the number of words $\bu$ satisfying~(i) and (ii) based on the tuples of $\mathcal{J}$.
	
	For this purpose, assume that the values $j_{k,i}^l$ form a fixed tuple of $\mathcal{J}$, i.e., $j_{k,i}^l$ satisfy~(i) and (ii). Next we count the number of words $\bu$ corresponding to these choices of $j_{k,i}^l$. The number of ways to choose the coordinates corresponding to $j_{k,i}^l$ within $\Pa_{k,i}$, where $l \in \{1, 2, \dots, \min\{q,k+1\}\}$, is equal to the multinomial coefficient
	\begin{equation} \label{Eq_mult_coefficient}
		\binom{p_{k,i}}{j_{k,i}^1, j_{k,i}^2, \dots, j_{k,i}^{ \min\{q,k+1\}}} \text.	
	\end{equation}
	Furthermore, if $k+1 \leq q$, there are $(q-k)^{j_{k,i}^{k+1}}$ choices for the symbols contributing to $j_{k,i}^{k+1}$. Therefore, if $q > s$, then $k+1 \leq q$ holds for all $k \in \{1, 2, \dots, \min\{q,s\}\}$ and, by going through all the choices for $k$ and $i$, we end up with the term in the formula for the case~(a). Thus, the claim~(a) follows by iterating over all the tuples of $\mathcal{J}$. For the case~(b) (when $q \leq s$), we notice that the condition $k+1 \leq q$ is satisfied for $k \in  [1,q-1]$. Furthermore, if $k \in  [q,\min\{q,s\}] = \{q\}$, i.e., $k = q$, then the value $j_{k,i}^{k+1}$ does not exist (since there are no symbols left in $[k,q-1] = \emptyset$) and the different choices within $\Pa_{k,i}$ can be counted similarly as in~\eqref{Eq_mult_coefficient}. Thus, the number of words $\bu$ corresponding to the current choice of $j_{k,i}^l$ is equal to 
	\[
	\prod_{k=1}^{q-1} \left(  \prod_{i=1}^{\cS(s,k)} \binom{p_{k,i}}{j_{k,i}^1, j_{k,i}^2, \dots, j_{k,i}^{k+1}}(q-k)^{j_{k,i}^{k+1}} \right)  \cdot \prod_{i=1}^{\cS(s,q)} \binom{p_{q,i}}{j_{q,i}^1, j_{q,i}^2, \dots, j_{q,i}^{q}} \text.
	\]
	Therefore, the claim~(b) also follows by iterating over all the tuples of $\mathcal{J}$.
\end{proof}

Note that although the formula of the theorem can seem somewhat complicated at first glance, it can actually be efficiently computed, for example, when $s$ is constant, as described in Remark~\ref{RmkComplexity}.
In the following remark, we observe that the formulas of the cases~(a)~and~(b) of the previous theorem can be combined.
\begin{remark}\label{remSimplifiedFormula}
    In the previous theorem, we can in fact use the formula
    \[
    |I_T(S)| = \sum_{\mathcal{J}} \left( \prod_{k=1}^{q-1} \left(  \prod_{i=1}^{\cS(s,k)} \binom{p_{k,i}}{j_{k,i}^1, j_{k,i}^2, \dots, j_{k,i}^{k+1}}(q-k)^{j_{k,i}^{k+1}} \right)  \cdot \prod_{i=1}^{\cS(s,q)} \binom{p_{q,i}}{j_{q,i}^1, j_{q,i}^2, \dots, j_{q,i}^{q}} \right)
    \]
    in both cases. Indeed, if $q \leq s$, then the claim immediately follows by the case~(b) of the theorem. Hence, we may assume that $q > s$ and it follows that
    \[
    \begin{split}
        \quad & \sum_{\mathcal{J}} \left( \prod_{k=1}^{q-1} \left(  \prod_{i=1}^{\cS(s,k)} \binom{p_{k,i}}{j_{k,i}^1, j_{k,i}^2, \dots, j_{k,i}^{k+1}}(q-k)^{j_{k,i}^{k+1}} \right)  \cdot \prod_{i=1}^{\cS(s,q)} \binom{p_{q,i}}{j_{q,i}^1, j_{q,i}^2, \dots, j_{q,i}^{q}} \right) \\  =&\sum_{\mathcal{J}} \left( \prod_{k=1}^{s} \left(  \prod_{i=1}^{\cS(s,k)} \binom{p_{k,i}}{j_{k,i}^1, j_{k,i}^2, \dots, j_{k,i}^{k+1}}(q-k)^{j_{k,i}^{k+1}} \right)\cdot \prod_{k=s+1}^{q-1} \left(  \prod_{i=1}^{\cS(s,k)} \binom{p_{k,i}}{j_{k,i}^1, j_{k,i}^2, \dots, j_{k,i}^{k+1}}(q-k)^{j_{k,i}^{k+1}} \right)  \cdot 1 \right) \\=&\sum_{\mathcal{J}} \left( \prod_{k=1}^{s} \left(  \prod_{i=1}^{\cS(s,k)} \binom{p_{k,i}}{j_{k,i}^1, j_{k,i}^2, \dots, j_{k,i}^{k+1}}(q-k)^{j_{k,i}^{k+1}} \right)  \cdot 1 \right) \text.
    \end{split}
    \]
    The equalities follow by observing that $\cS(s,k)=0$ for $k>s$ and that an empty product, where the lower bound is greater than the upper bound, is equal to $1$. Thus, the claim also follows in the case with $q > s$.
\end{remark}

We note that in the binary case $q=2$, we may greatly simplify the formula \eqref{Eq1stFormula} with the help of the fact that $\mathcal{S}(s,1)=1$ for each $s\geq1$.
\begin{corollary} \label{Cor_bin_intersection_size}
	Let $S = \{\bc^1, \bc^2, \dots, \bc^s\}$ be a subset of $\Z^n_2$ and $T = (t_1, t_2, \ldots, t_s)$ a list of positive integers, where $s\geq2$. Let $j_{k,i}^l$ be nonnegative integers, where $k \in \{1, 2\}$, $i \in \{1, 2, \dots, \cS(s,k)\}$ and $l \in \{1, 2\}$, satisfying the following conditions: (i) $\sum_{l=1}^{2} j_{k,i}^l = p_{k,i}$ and (ii) for each $\bc^m \in S$,
	\[
	\sum_{k=1}^{2} \sum_{i=1}^{\cS(s,k)} \left( p_{k,i} - j_{k,i}^{f_{k,i}(\bc^m)} \right) \leq t_m \text.
	\]
	Let $\mathcal{J}$ denote all the tuples formed by $j_{k,i}^l$ which satisfy the conditions~(i)~and~(ii). The formula for the size of the intersection can be presented as:
		\begin{equation} \label{Eq1stFormulaBin}
		    |I_T(S)| = \sum_{\mathcal{J}} \left( \binom{p_{1,1}}{j_{1,1}^1}   \cdot \prod_{i=1}^{\cS(s,2)} \binom{p_{2,i}}{j_{2,i}^1} \right)  \text,
		\end{equation}
	    where the sum goes through all the tuples of $\mathcal{J}$.
\end{corollary}

Observe that if $n$, $q$, $S$ and $T$ are fixed, then by the theorem the size of the intersection  $I_T(S)$
depends only on the pseudo-partition of the coordinate positions formed by $\Pa_{k,i}$ and especially on their sizes $p_{k,i}$ (backing up the former claim that the values $p_{k,i}$ are essential in determining the size of the intersection). Recall that if $q \geq 2$ and $|S| = 2$, then the exact cardinality of $I_t(S)$ has been determined in~\cite[Corollary~1]{Levenshtein}. In the following corollary, we extend this result to the case with varying radii $t_m$. Recall that the corollary could also be derived based on the intersection $S_{t_1}(\bc^1) \cap S_{t_2}(\bc^2)$ of the spheres as briefly discussed in Section~\ref{SecRelatedWork}; see, for example, \cite[Theorem 14.2.2]{blahut1983theory} concerning the intersection numbers of the Hamming scheme. 
\begin{corollary} \label{Cor2balls}
	Let $S = \{\bc^1, \bc^2\}$ be a subset of $\Z_q^n$, $T = (t_1, t_2)$ a list of positive integers and $d$ an integer such that $d = d(\bc^1, \bc^2)$.
	\begin{itemize}
		\item[(I)] If $q=2$, then
		\[
		|I_T(S)| = \sum_{i_1=0}^{\lfloor (t_1+t_2-d)/2 \rfloor} \left( \binom{n-d}{i_1 }\sum_{i_2=i_1+d-t_2}^{t_1-i_1} \binom{d}{i_2} \right) \text. 
		\]
		\item[(II)] If $q \geq 3$, then the cardinality of the intersection $I_T(S)$ is equal to
		\[
		\sum_{i_1=0}^{\lfloor (t_1+t_2-d)/2 \rfloor} \left( \binom{n-d}{i_1}(q-1)^{i_1} \sum_{i_2=i_1+d-t_1}^{t_2-i_1} \sum_{i_3=i_1+d-t_2}^{t_1-i_1} \binom{d}{i_2, i_3, d-(i_2+i_3)} (q-2)^{d-i_2-i_3} \right) \text. 
		\]
	\end{itemize}
\end{corollary}
\begin{proof}
	Observe first that by~\eqref{Eq_Pqs}, we have $P_2(2) = 2$; in particular, the pseudo-partition of the coordinate positions is formed by $\Pa_{1,1}$ and $\Pa_{2,1}$, which respectively correspond to the partitions $P_{1,1} = (\{1,2\})$ and $P_{2,1} = (\{1\}, \{2\})$ of $[1,2]$ (where the list notation emphasizes that the blocks are listed in the fixed order declared above). It is immediate that $p_{2,1} = d = d(\bc^1, \bc^2)$ and $p_{1,1} = n - p_{2,1}$. For the rest of the proof, we may assume without loss of generality that $\bc^1 = \zero$ and $\bc^2 = 0^{p_{1,1}}1^{p_{2,1}}$. Hence, we have $\Pa_{1,1} = [1,p_{1,1}]$ and $\Pa_{2,1} = [p_{1,1}+1,n]$. The proof now divides into the following two cases depending on whether (I) $q = 2$ or (II) $q \geq 3$.
	
	(I) Assume first that $q=2$. Let the nonnegative integers $j_{1,1}^1$, $j_{1,1}^2$, $j_{2,1}^1$ and $j_{2,1}^2$, as defined in Corollary~\ref{Cor_bin_intersection_size} (or Theorem~\ref{Thm_intersection_size}) satisfy the following conditions: (i) $j_{1,1}^1 + j_{1,1}^2 = p_{1,1}$ and $j_{2,1}^1 + j_{2,1}^2 = p_{2,1}$ as well as (ii) $(p_{1,1} - j_{1,1}^1) + (p_{2,1} - j_{2,1}^1) = j_{1,1}^2 + j_{2,1}^2 \leq t_1$ and $(p_{1,1} - j_{1,1}^1) + (p_{2,1} - j_{2,1}^2) = j_{1,1}^2 + j_{2,1}^1 \leq t_2$. Furthermore, denoting $j_{1,1}^2 = i_1$ and $j_{2,1}^2 = i_2$, the inequalities~(ii) can be written as follows: $i_1 + i_2 \leq t_1$ and $i_1 + (d-i_2) \leq t_2$. Eliminating first $i_2$ using the well-known Fourier-Motzkin elimination method~\cite{FMElimination}, we obtain $i_1 + d - t_2 \leq i_2 \leq t_1 - i_1$ and $i_1 + d - t_2 \leq t_1 - i_1$, which finally imply $i_1 + d - t_2 \leq i_2 \leq t_1 - i_1$ and $i_1 \leq (t_1+t_2-d)/2$. Therefore, the claim follows by \eqref{Eq1stFormulaBin} in Corollary~\ref{Cor_bin_intersection_size} 
    when we take into account the equalities~(i) and the fact that $j_{k,i}^l$ (or $i_1$ and $i_2$) are nonnegative integers, that is, 
    $\mathcal{J}=\{(n-d-i_1,i_1,d-i_2,i_2)\mid 0\le i_1\le \lfloor (t_1+t_2-d)/2 \rfloor,i_1+d-t_2\le i_2 \le t_1-i_1   \}.$

	(II) Assume then that $q \geq 3$. Let the nonnegative integers $j_{1,1}^1$, $j_{1,1}^2$, $j_{2,1}^1$, $j_{2,1}^2$ and $j_{2,1}^3$ of Theorem~\ref{Thm_intersection_size} satisfy the following conditions: (i) $j_{1,1}^1 + j_{1,1}^2 = p_{1,1}$ and $j_{2,1}^1 + j_{2,1}^2 + j_{2,1}^3 = p_{2,1}$ as well as (ii) $(p_{1,1} - j_{1,1}^1) + (p_{2,1} - j_{2,1}^1) = j_{1,1}^2 + (d- j_{2,1}^1) \leq t_1$ and $(p_{1,1} - j_{1,1}^1) + (p_{2,1} - j_{2,1}^2) = j_{1,1}^2 + (d - j_{2,1}^2) \leq t_2$. Furthermore, denoting $j_{1,1}^2 = i_1$, $j_{2,1}^1 = i_2$, $j_{2,1}^2 = i_3$ and $j_{2,1}^3 = i_4$ the inequalities~(ii) can be written as follows: $i_1 + (d- i_2) \leq t_1$ and $i_1 + (d - i_3) \leq t_2$. Moreover, taking into account the equalities~(i), the inequalities~(ii) can be formulated as $i_1 + i_3 + i_4 \leq t_1$ and $i_1 + i_2 + i_4 \leq t_2$, which further imply $i_1 + i_3 \leq t_1$ and $i_1 + i_2 \leq t_2$. Therefore, eliminating $i_2$ and $i_3$ from the inequalities 
    \[
    \left\{
    \setlength\arraycolsep{0pt}
    \begin{array}{ r @{} r  >{{}}c<{{}} r  >{{}}c<{{}}  r }
    i_1 + (d- i_2) &~\leq t_1 \\
    i_1 + (d - i_3) &~\leq t_2   \\
    i_1 + i_3 &~\leq t_1 \\ 
    i_1 + i_2 &~\leq t_2       \\
    \end{array}
    \right.
    \]
       by the Fourier-Motzkin method, we obtain that 
       \[
    \left\{
    \setlength\arraycolsep{0pt}
    \begin{array}{ r @{} r  >{{}}c<{{}} r  >{{}}c<{{}}  r }
    i_1 + d - t_1 &~\leq i_2 &\leq t_2 - i_1& \\
    i_1 + d - t_2 &~\leq i_3 &\leq t_1 - i_1&   \\
    i_1 + d - t_1 & &\leq t_2 - i_1& \\ 
    i_1 + d - t_2 & &\leq t_1 - i_1&       \\
    \end{array}
    \right. \text.
    \]
   Furthermore, the latter two inequalities are both equivalent with $i_1 \leq (t_1+t_2-d)/2$.  Therefore, the claim follows by \eqref{Eq2ndFormula} when we take into account the equalities~(i) and the fact that $j_{k,i}^l$ (or $i_1$, $i_2$ and $i_3$) are nonnegative integers.
\end{proof}
Notice that if we choose $t_1 = t_2 = t$ in the previous corollary, then Corollary~1 of~\cite{Levenshtein} immediately follows. Consider then the case with $q = 2$ and $|S| = 3$. Previously, in~\cite{chll}, it has been shown that then the cardinality of $I_T(S)$ is based solely on the pairwise distances between the words of $S \ (\subseteq \Z_2^n)$. Moreover, in~\cite[Theorem~6]{yaakobi2018uncertainty}, the maximum cardinality $N_t^2(3,d)$ has been determined when $n$ is large enough. In the following corollary of Theorem~\ref{Thm_intersection_size} (or Corollary~\ref{Cor_bin_intersection_size}), we extend these results by presenting an exact formula for the size of the intersection $I_T(S)$, where $T = (t_1, t_2, t_3)$.
\begin{corollary} \label{Cor3balls}
    Let $S = \{\bc^1, \bc^2, \bc^3\}$ be a subset of $\Z^n_2$, $T = (t_1, t_2, t_3)$ a list of positive integers as well as $d_1$, $d_2$ and $d_3$ integers such that $d_1 = d(\bc^1, \bc^2)$, $d_2 = d(\bc^1, \bc^3)$ and $d_3 = d(\bc^2, \bc^3)$. If
    \[
    \left\{
	\begin{aligned}
		p_{2,1} &= \left( -d_1 + d_2 + d_3 \right)/2 \\
		p_{2,2} &= \left(  d_1 - d_2 + d_3 \right)/2 \\
		p_{2,3} &= \left(  d_1 + d_2 - d_3 \right)/2
	\end{aligned}
	\right.
    \]
    and $p_{1,1} = n - p_{2,1} - p_{2,2} - p_{2,3}$, then the cardinality of the intersection $I_T(S)$ is equal to
    \begin{equation} \label{Eq_3ball_intersection}
        |I_T(S)| = \sum_{i_1, i_2, i_3, i_4} \binom{p_{1,1}}{i_1}\binom{p_{2,1}}{i_2}\binom{p_{2,2}}{i_3}\binom{p_{2,3}}{i_4} \text,
    \end{equation}
    where the integers $i_1$, $i_2$, $i_3$ and $i_4$ satisfy the following conditions: 
    \begin{itemize}
        \item[(a)] $0 \leq i_1 \leq p_{1,1}$,
        
        \item[(b)] $0 \leq i_4 \leq -i_1 - (p_{2,1} + p_{2,2})/2 + (t_2+t_3)/2$,
        
        \item[(c)] $i_1 + (p_{2,1} + 2p_{2,2} + p_{2,3})/2 - (t_1+t_3)/2 \leq i_3 \leq p_{2,2}$ and
        
        \item[(d)] $\max\{i_1-i_3-i_4+p_{2,1}+p_{2,2}+p_{2,3}-t_1, i_1+i_3+i_4+p_{2,1}-t_2\} \leq i_2 \leq -i_1+i_3-i_4-p_{2,2}+t_3$.
    \end{itemize}
\end{corollary}
\begin{proof}
    Observe first that by~\eqref{Eq_Pqs}, we have $P_2(3) = \cS(3,1) + \cS(3,2) = 1 + 3 = 4$; in particular, the pseudo-partition of the coordinate positions is formed by $\Pa_{1,1}$, $\Pa_{2,1}$, $\Pa_{2,2}$ and $\Pa_{2,3}$ with $P_{1,1} = (\{1,2,3\})$, $P_{2,1} = (\{1,2\}, \{3\})$, $P_{2,2} =  (\{1,3\}, \{2\})$ and $P_{2,3} = (\{1\}, \{2,3\})$. Taking into account the isomorphisms of $\Z_2^n$, we may without loss of generality assume for the rest of the proof that $\bc^1 = \zero$, $\bc^2 = 0^{p_{1,1} + p_{2,1}}1^{p_{2,2} + p_{2,3}}$ and $\bc^3 = 0^{p_{1,1}}1^{p_{2,1}}0^{p_{2,2}}1^{p_{2,3}}$. Based on the pairwise distances $d_1 = d(\bc^1, \bc^2)$, $d_2 = d(\bc^1, \bc^3)$ and $d_3 = d(\bc^2, \bc^3)$ as well as $p_{k,i}$ we obtain the following system of linear equalities:
	\[
	\left\{
	\begin{alignedat}{3}
		        &  {} & p_{2,2} & +{} & p_{2,3} & = d_1 \\
		p_{2,1} &  {} &         & +{} & p_{2,3} & = d_2 \\
		p_{2,1} & +{} & p_{2,2} &  {} &         & = d_3
	\end{alignedat}
	\right. \ \iff \
	\left\{
	\begin{aligned}
		p_{2,1} &= \left( -d_1 + d_2 + d_3 \right)/2 \\
		p_{2,2} &= \left(  d_1 - d_2 + d_3 \right)/2 \\
		p_{2,3} &= \left(  d_1 + d_2 - d_3 \right)/2\text.
	\end{aligned}
	\right. 
	\]
	Based on the solution, the values $p_{2,1}$, $p_{2,2}$ and $p_{2,3}$ can be expressed in terms of the pairwise distances. Furthermore, since $p_{1,1} = n - p_{2,1} - p_{2,2} - p_{2,3}$, that is also the case with $p_{1,1}$. Therefore, the size of the intersection $I_T(S)$ can be calculated (by 
    Corollary~\ref{Cor_bin_intersection_size}) solely based on $d_1$, $d_2$ and $d_3$.

    In order to prove the formula~\eqref{Eq_3ball_intersection}, we next consider which values are possible for $j_{k,i}^l$ of 
    Corollary~\ref{Cor_bin_intersection_size}. For this purpose, we denote $j_{1,1}^2 = i_1$, $j_{2,1}^1 = i_2$, $j_{2,2}^1 = i_3$ and $j_{2,3}^1 = i_4$. The nonnegative integers $j_{k,i}^l$ satisfy the following conditions of Corollary~\ref{Cor_bin_intersection_size}:
    (i) $j_{k,i}^1 + j_{k,i}^2 = p_{k,i}$ for each $(k,i) \in \{(1,1), (2,1), (2,2), (2,3)\}$ and (ii) $i_1 + (p_{2,1} - i_2) + (p_{2,2} - i_3) + (p_{2,3} - i_4) \leq t_1$, $i_1 + (p_{2,1} - i_2) + i_3 + i_4 \leq t_2$ and $i_1 + i_2 + (p_{2,2} - i_3) + i_4 \leq t_3$. First, eliminating $i_2$ using the Fourier-Motzkin elimination method, we obtain the following system of linear inequalities equivalent to~(ii): 
    \begin{itemize}
        \item $\max\{i_1-i_3-i_4+p_{2,1}+p_{2,2}+p_{2,3}-t_1, i_1+i_3+i_4+p_{2,1}-t_2\} \leq i_2 \leq -i_1+i_3-i_4-p_{2,2}+t_3$,
        
        \item $i_1-i_3-i_4+p_{2,1}+p_{2,2}+p_{2,3}-t_1 \leq -i_1+i_3-i_4-p_{2,2}+t_3$ and 
        
        \item $i_1+i_3+i_4+p_{2,1}-t_2 \leq -i_1+i_3-i_4-p_{2,2}+t_3$.
    \end{itemize}
    Furthermore, the latter two inequalities can be written equivalently as follows: $i_1 + (p_{2,1} + 2p_{2,2} + p_{2,3})/2 - (t_1+t_3)/2 \leq i_3$ and $i_4 \leq -i_1 - (p_{2,1} + p_{2,2})/2 + (t_2+t_3)/2$. Finally, when we take into account the equalities~(i) and the fact that $i_1$, $i_2$, $i_3$ and $i_4$ are nonnegative integers, the conditions~(a)--(d) follow as well as the claim.
\end{proof}

In the following remark, we compare the previous corollary to the results of~\cite{yaakobi2018uncertainty}.
\begin{remark} \label{Rem3balls}
    In~\cite[Lemma~4]{yaakobi2018uncertainty}, it has been shown that if $\bc^1, \bc^2, \bc^3 \in \Z_2^n$ are such that $d(\bc^1, \bc^2) = d(\bc^1, \bc^3) = d(\bc^2, \bc^3) = d$ and $d$ is even, then
    \[
    |I_t(S)| = \sum_{i_1, i_2, i_3, i_4} \binom{n-3d/2}{i_1}\binom{d/2}{i_2}\binom{d/2}{i_3}\binom{d/2}{i_4} \text,
    \]
    where $i_1$, $i_2$, $i_3$ and $i_4$ satisfy the following constraints: \begin{itemize}
        \item[(A)] $0 \leq i_1 \leq t - d/2$, \item[(B)] $i_1 + d/2 - t \leq i_4 \leq t - d/2 - i_1$, 
        \item[(C)] $d - t + i_1 \leq i_3 \leq t - (i_1 + i_4)$ and 
        \item[(D)] $\max\{i_1-i_3-i_4+3d/2-t, i_1+i_3+i_4+d/2-t\} \leq i_2 \leq t - (i_1+i_4+d/2-i_3)$.
    \end{itemize}Now we give the following observations concerning the constraints (a)--(d) of Corollary~\ref{Cor3balls} and (A)--(D):
    \begin{itemize}
        \item The condition~(a) of Corollary~\ref{Cor3balls} can be replaced by $0 \leq i_1 \leq (t_2+t_3)/2 - (p_{2,1} + p_{2,2})/2$ since if $i_1 > (t_2+t_3)/2 - (p_{2,1} + p_{2,2})/2$, then $(p_{2,1} + p_{2,2})/2 + (t_2+t_3)/2 - i_1 < 0$ and no nonnegative integer $i_4$ satisfies the condition~(b). 

        \item The constraint~(B) can be replaced by $0 \leq i_4 \leq t - d/2 - i_1$ since $i_1 + d/2 - t \leq 0$ by~(A).

        \item The constraint~(C) can be replaced by $d - t + i_1 \leq i_3 \leq d/2$ since $d/2 \leq t - (i_1 + i_4)$ by~(B) (and $i_3$ is trivially upper bounded by $d/2$ as $\binom{d/2}{i_3}$ is assumed to be equal to $0$ for $i_3 > d/2$).
    \end{itemize}
    Thus, if we substitute $t_1 = t_2 = t_3 = t$ and $d_1 = d_2 = d_3 = d$ in Corollary~\ref{Cor3balls} as well as take the previous observations into account, then Lemma~4 of~\cite{yaakobi2018uncertainty} is immediately obtained.
    
    Similarly, Lemma~5 of~\cite{yaakobi2018uncertainty}, where $\bc^1, \bc^2, \bc^3 \in \Z_q^n$ are such that $d(\bc^1, \bc^2) = d(\bc^1, \bc^3) = d$, $d(\bc^2, \bc^3) = d+1$ and $d$ is odd, also follows by our corollary.
    \end{remark}

\subsection{Additional Information Required for Calculating the Cardinality of the Intersection} \label{subsecNumPara}

Recall (from~\cite{Iiro}) that if $q=2$ and $|S| \geq 4$ or $q \geq 3$ and $|S| \geq 3$, then the size of the intersection $I_t(S)$ is no longer determined only by the pairwise distances of the words in $S$. With respect to Theorem~\ref{Thm_intersection_size}, this is due to the fact that the values $p_{k,i}$ do not depend only on the pairwise distances, as is further discussed in the following example. 
\begin{example} \label{Ex_1auxparameter}
	Let $q \geq 3$. Consider the intersections $I_2(S)$ and $I_2(S')$, where $S = \{0000, 1100, 0110\}$ and $S' = \{0000,1100,2200\}$. Recalling the notation of Example~\ref{Ex_notation}, we have $p_{1,1} = p_{2,1} = p_{2,2} = p_{2,3} = 1$ and $p_{3,1} = 0$ for the set $S$. Similarly, it can be obtained for $S'$ that $p_{1,1} = p_{3,1} = 2$ and $p_{2,1} = p_{2,2} = p_{2,3} = 0$. Thus, the result of Example~\ref{exSubInt} stating that the sizes of $I_2(S)$ and $I_2(S')$ are different is consistent with Theorem~\ref{Thm_intersection_size}; in fact, the exact values $|I_2(S)|=4q-2$ and $|I_2(S')|=q^2$ could also be obtained by the theorem (although the \emph{ad hoc} method of the example is simpler).
	
	Consider then an intersection of four binary Hamming balls. As in~\cite[Example~2.4.12]{chll}, let $U = \{\bu^1, \bu^2, \bu^3, \bu^4\}$ and $V = \{\bv^1, \bv^2, \bv^3, \bv^4\}$, where $\bu^1 = \bv^1 = 0000$, $\bu^2 = \bv^2 = 1100$, $\bu^3 = \bv^3 = 1010$, $\bu^4 = 0110$ and $\bv^4 = 1001$. All the pairwise distances of the words in $U$ and $V$ are equal to $2$. However, by~\cite{chll}, we have $|I_2(U)| = 4 \neq 5 = |I_2(V)|$. It is straightforward to verify that the values $p_{k,i}$ corresponding to $U$ and $V$ are different; for example, $p_{1,1} = 1$ for $U$ and $p_{1,1} = 0$ for $V$. Moreover, the sizes of the intersections $I_2(U)$ and $I_2(V)$ could also be computed according to Theorem~\ref{Thm_intersection_size}.
\end{example}

In conclusion, if $q=2$ and $|S| \geq 4$, or $q \geq 3$ and $|S| \geq 3$, then in order to determine all $p_{k,i}$, some additional information besides the pairwise distances of the words in $S$ is required. (However, for certain $S$ and $T$, it might be possible to determine the size of the intersection $I_T(S)$ without any additional information; in particular, in the trivial case when $|I_T(S)| = 0$.) Notice also that the required additional information does not depend on whether we are considering the case with varying radii $t_m$ or the case with all the radii being equal to $t$. In what follows, we focus on the required additional information; in particular, we show that if $q=2$ and $|S| = 4$ or $q \geq 3$ and $|S| = 3$, then one additional parameter is enough to determine the size of the intersection. For this purpose, let $S = \{\bc^1, \bc^2, \dots, \bc^s\}$ be a subset of $\Z^n_q$. Recall that the sets $\Pa_{k,i}$ form a pseudo-partition of the coordinate positions $[1,n]$. Denote the set of all pairs of $S$, i.e., all subsets of size $2$, by $S[2]$. In order to express $d(\bc^j, \bc^{j'})$ as a sum of the values $p_{k,i}$, we define an auxiliary function $g_{k ,i}: S[2] \to \{0,1\}$, for each $k \in \{1, 2, \dots, \min\{q,s\}\}$ and $i \in \{1, 2, \dots, \cS(s,k)\}$, as follows:
\[
g_{k ,i}(\bc, \bc') =
\begin{cases}
	0 \text, & \text{if $\ind(\bc)$ and $\ind(\bc')$ belong to the same block of the partition } P_{k,i} \text; \\
	1 \text, & \text{if $\ind(\bc)$ and $\ind(\bc')$ belong to different blocks of the partition } P_{k,i} \text.
\end{cases}
\]
Observe that $g_{1,1}(\bc, \bc') = 0$ for all $\bc, \bc' \in S$ since the partition $P_{1,1} = \{[1,s]\}$. Hence, the distance of the words $\bc$ and $\bc'$ belonging to $S$ can now be written as follows:
\begin{equation} \label{Eq_linear_system}
	d(\bc, \bc') = \sum_{k=2}^{\min\{q,s\}} \sum_{i=1}^{\cS(s,k)} g_{k ,i}(\bc, \bc') \cdot p_{k,i} \text.
\end{equation}
Indeed, the value $p_{k,i}$ contributes to the distance $d(\bc, \bc')$ if and only if the coordinates of $\bc$ and $\bc'$ disagree within $\Pa_{k,i}$. Since each pair of words belonging to $S[2]$ corresponds to an equation of type~\eqref{Eq_linear_system}, we may form a system of $\binom{s}{2}$ linear equalities with $P_q(s)-1$ variables $p_{k,i}$ (excluding $p_{1,1}$); for convenience, denote the system by $\mathcal{L}$. Furthermore, we assume that the pairs of $S[2]$ are listed in some fixed order and so are the partitions $P_{k,i}$; the chosen order does not matter, but for the later discussions it is helpful that it is fixed. Notice that the number of variables $P_q(s)-1$ is larger (even significantly larger, as discussed later in Remark~\ref{Rmk_asymptotics}) than the number of linear equalities $\binom{s}{2}$, when $q=2$ and $|S| \geq 4$ or $q \geq 3$ and $|S| \geq 3$; in particular, for all distinct $j_1, j_2 \in [1,s]$, the partition $\{\{j_1, j_2\}, [1,s] \setminus \{j_1, j_2\}\}$ into two blocks contributes to the value $P_q(s)-1$ implying that $P_q(s)-1 \geq \binom{s}{2}$. 

In what follows, we focus more closely on the linear independence of the system $\mathcal{L}$. For this purpose, let $\bp = (p_{k,i})_{(P_q(s)-1) \times 1}$ and $\bd = (d(\bc,\bc'))_{\binom{s}{2} \times 1}$ be column vectors, where $\bp$ consists of all $p_{k,i}$'s except $p_{1,1}$. Furthermore, define an $\binom{s}{2} \times (P_q(s)-1)$ matrix $M_{q,s}$ as follows: each row of the matrix corresponds to a pair $\{\bc, \bc'\} \in S[2]$ and consists of the coefficients $g_{k ,i}(\bc, \bc')$ of~\eqref{Eq_linear_system}. The construction of the matrix $M_{q,s}$ is illustrated in the following example.
\begin{example} \label{ExM2s}
    In what follows, we construct the matrices $M_{q,s}$ for $q = 2$ and $s \in \{2,3,4\}$.
    \begin{itemize}
        \item For $s = 2$, the only pair of $S[2]$ is $\{1,2\}$ and the only partition corresponding to the columns is $P_{2,1} = \{\{1\}, \{2\}\}$. Hence, we have $M_{2,2} = (1)$.
        \item For $s=3$, we fix the order of the pairs of $S[2]$ as $\{1,2\}$, $\{1,3\}$ and $\{2,3\}$ as well as the order of the partitions $P_{2,i}$ as in Example~\ref{Ex_notation}: $P_{2,1} = \{\{1,2\}, \{3\}\}$, $P_{2,2} =  \{\{1,3\}, \{2\}\}$ and $P_{2,3} = \{\{1\}, \{2,3\}\}$. Now we have
        \[
        M_{2,3} = 
        \begin{pmatrix}    
            0 & 1 & 1 \\
            1 & 0 & 1 \\
            1 & 1 & 0
        \end{pmatrix} \text.
        \]
        \item Assume that $s = 4$. Now $S[2] = \{\{1,2\}, \{1,3\}, \{1,4\}, \{2,3\}, \{2,4\}, \{3,4\}\}$, where the pairs are listed in the chosen (fixed) order. Furthermore, we fix the order of the partitions $P_{2,i}$ as follows: $P_{2,1} = \{\{1\},\{2,3,4\}\}$, $P_{2,2} = \{\{1,2\},\{3,4\}\}$, $P_{2,3} = \{\{1,3,4\},\{2\}\}$, $P_{2,4} = \{\{1,2,3\},\{4\}\}$, $P_{2,5} = \{\{1,4\},\{2,3\}\}$, $P_{2,6} = \{\{1,2,4\},\{3\}\}$ and $P_{2,7} = \{\{1,3\},\{2,4\}\}$. Now we have
        \[
        M_{2,4} = 
        \begin{pmatrix}    
            1 & 0 & 1 & 0 & 1 & 0 & 1 \\
            1 & 1 & 0 & 0 & 1 & 1 & 0 \\
            1 & 1 & 0 & 1 & 0 & 0 & 1 \\
            0 & 1 & 1 & 0 & 0 & 1 & 1 \\
            0 & 1 & 1 & 1 & 1 & 0 & 0 \\
            0 & 0 & 0 & 1 & 1 & 1 & 1
        \end{pmatrix} \text.
        \]
    \end{itemize}
\end{example}

The system $\mathcal{L}$ of linear equations can now be expressed in a matrix form as follows: $M_{q,s} \cdot \bp = \bd$. The linear equations of $\mathcal{L}$ are linearly independent if and only if the rows of $M_{q,s}$ are. In the following lemma, we show that this is indeed the case for $M_{q,s}$ and present a proof that is based on a connection to the \emph{Johnson graph} $J(s,2)$. The graph $J(s,2)$ has the vertex set $V(J(s,2)) = S[2]$, and two vertices $\{j_1,j_2\}$ and $\{j'_1,j'_2\}$ are adjacent if $|\{j_1,j_2\} \cap \{j'_1,j'_2\}| = 1$. For more details on the Johnson graph $J(s,2)$, the interested reader is referred to~\cite[Section~9.1]{BrouweretalDRG}.

\begin{lemma} \label{Lemma_lin_independence}
	The rows of the matrix $M_{q,s}$ are linearly independent.
\end{lemma}
\begin{proof}
    Observe first that for $s \in \{2,3,4\}$, the rows of the $\binom{s}{2} \times (P_q(s)-1)$ matrices $M_{q,s}$ are linearly independent. Indeed, the submatrices $M_{2,s}$ (constructed in Example~\ref{ExM2s}) of $M_{q,s}$ have full rank since $\det(M_{2,2}) = 1$, $\det(M_{2,3}) = 2$ and $\det(M'_{2,4}) = -8$, where $M'_{2,4}$ is obtained from $M_{2,4}$ by deleting its last column. Therefore, it follows that the rows of $M_{q,s}$ are also linearly independent. 
    Hence, we may assume that $s \geq 5$. Consider then the columns of $M_{q,s}$ corresponding to the partitions $\{\{j_1, j_2\}, [1,s] \setminus \{j_1, j_2\}\}$, where $j_1$ and $j_2$ are distinct integers of $[1,s]$, and based on these columns, construct an $\binom{s}{2} \times \binom{s}{2}$ submatrix of $M_{q,s}$, denoted by $M'_{q,s}$. In what follows, we show that $M'_{q,s}$ is equal to an adjacency matrix of the Johnson graph $J(s,2)$.
	
	Observe that the adjacency of the distinct vertices $\{j_1,j_2\}$ and $\{j'_1,j'_2\}$ of $J(s,2)$ can be equivalently formulated as follows: $|\{j_1,j_2\} \cap \{j'_1,j'_2\}| = 1$ if and only if the integers $j_1$ and $j_2$ belong to different subsets of the partition $\{\{j'_1, j'_2\}, [1,s] \setminus \{j'_1, j'_2\}\}$. Thus, an adjacency matrix of $J(s,2)$ is equal to $M'_{q,s}$. By~\cite[Theorem~9.1.2]{BrouweretalDRG}, the eigenvalues of $M'_{q,s}$ are $\lambda_0 = 2(s-2)$, $\lambda_1 = s-4$ and $\lambda_2 = -2$. Therefore, as $s > 4$, the eigenvalues are non-zero, implying that so is the determinant of $M'_{q,s}$. This implies that the rows of $M_{q,s}$ are linearly independent. Thus, the claim follows.
\end{proof}

Based on the previous lemma, we may observe that the number of free variables in the solution of $\mathcal{L}$ is equal to $P_q(s) - 1 - \binom{s}{2}$. In the discussions after Example~\ref{Ex_1auxparameter}, we claimed that if $q=2$ and $|S| = 4$ or $q \geq 3$ and $|S| = 3$, then one additional parameter alongside the pairwise distances is enough. (Observe that in both cases $P_q(s) - 1 - \binom{s}{2} = 1$.) In the following theorem, we prove the claim by showing that any of the values $p_{k,i}$ can be chosen as a free variable in these cases. 
\begin{theorem}
    Let all the pairwise distances $d(\bc^i, \bc^{j})$ be known for distinct $\bc^i, \bc^j \in S$. If $q=2$ and $|S| = 4$ or $q \geq 3$ and $|S| = 3$, then knowing any one of the values $p_{k,i}$ is enough to determine all of them.
\end{theorem}
\begin{proof}
    Assume first that $q=2$ and $|S| = 4$. Recall from Example~\ref{ExM2s} that 
    \[
        M_{2,4} = 
        \begin{pmatrix}    
            1 & 0 & 1 & 0 & 1 & 0 & 1 \\
            1 & 1 & 0 & 0 & 1 & 1 & 0 \\
            1 & 1 & 0 & 1 & 0 & 0 & 1 \\
            0 & 1 & 1 & 0 & 0 & 1 & 1 \\
            0 & 1 & 1 & 1 & 1 & 0 & 0 \\
            0 & 0 & 0 & 1 & 1 & 1 & 1
        \end{pmatrix} \text.
    \]
    It is straightforward to verify that deleting any column of $M_{2,4}$ results in a $6 \times 6$ submatrix of rank $6$. Therefore, any of the values $p_{k,i}$ can be chosen as a free variable in the system of equalities $\mathcal{L}$ corresponding to $M_{2,4}$. Thus, for determining the values $p_{k,i}$, it is enough to know all the pairwise distances of the centers $\bc^i$ as well as one additional parameter, which can be any of $p_{k,i}$.

    Assume then that $q \geq 3$ and $|S| = 3$. As in Example~\ref{ExM2s}, we may construct
    \[
        M_{q,3} = M_{3,3} = 
        \begin{pmatrix}    
            0 & 1 & 1 & 1\\
            1 & 0 & 1 & 1\\
            1 & 1 & 0 & 1
        \end{pmatrix} \text.
    \]
    As above, deleting any column of $M_{3,3}$ results in a $3 \times 3$ submatrix of rank $3$. Thus, besides the pairwise distances knowing any of $p_{k,i}$ is enough for determining all of them.
\end{proof}

It should be noted that in general, when $q \geq 3$ and $|S| \geq 4$ or $q \geq 2$ and $|S| \geq 5$, not all sets of $P_q(s) - 1 - \binom{s}{2}$ values of $p_{k,i}$ can be chosen as free variables. For example, if $q=3$ and $|S|=4$, then $P_3(4) - 1 = 13$ and $\binom{4}{2} = 6$ as well as
\[ 
M_{q,s} = M_{3,4} = 
\begin{pmatrix} 
1 & 0 & 1 & 0 & 1 & 0 & 1 & 1 & 1 & 1 & 0 & 1 & 1 \\ 
1 & 1 & 0 & 0 & 1 & 1 & 0 & 1 & 1 & 1 & 1 & 0 & 1 \\ 
1 & 1 & 0 & 1 & 0 & 0 & 1 & 1 & 1 & 1 & 1 & 1 & 0 \\ 
0 & 1 & 1 & 0 & 0 & 1 & 1 & 1 & 0 & 1 & 1 & 1 & 1 \\ 
0 & 1 & 1 & 1 & 1 & 0 & 0 & 1 & 1 & 0 & 1 & 1 & 1 \\
0 & 0 & 0 & 1 & 1 & 1 & 1 & 0 & 1 & 1 & 1 & 1 & 1
\end{pmatrix} \text.
\]
The submatrix
\[ 
\begin{pmatrix} 
1 & 0 & 1 & 0 & 1 & 1 \\ 
1 & 1 & 0 & 0 & 1 & 1 \\ 
1 & 1 & 0 & 1 & 0 & 1 \\ 
0 & 1 & 1 & 0 & 0 & 1 \\ 
0 & 1 & 1 & 1 & 1 & 1 \\ 
0 & 0 & 0 & 1 & 1 & 0
\end{pmatrix} 
\]
of $M_{3,4}$, where the columns $1$, $2$, $3$, $4$, $5$ and $8$ have been chosen, has determinant equal to $0$. Therefore, the set of parameters corresponding to the $P_3(4) - 1 - \binom{4}{2} = 7$ omitted columns cannot be chosen as a set of free variables. However, choosing the first six columns leads to an invertible submatrix and, hence, $p_{k,i}$ corresponding to the columns $7$--$13$ can be chosen as free variables. Moreover, in general, using standard Gaussian elimination one can find a suitable choice of variables $P_q(s) - 1 - \binom{s}{2}$ such that all of $p_{k,i}$ can be uniquely determined (as indicated by Lemma~\ref{Lemma_lin_independence}).

\smallskip

In Subsection~\ref{sec3balls}, the case with $q \geq 3$ and $S = \{\bc^0, \bc^1, \bc^2\}$ will be further studied. As seen in the remark, we require now one parameter in addition to the pairwise distances between the words in $S$. In Subsection~\ref{sec3balls}, with the assumption of $\bc^0$ being the all-zero word, we consider the additional parameter $A = |\supp(\bc^1)\cap \supp(\bc^2)|$, i.e., $A$ denotes the number of coordinate positions, where both $\bc^1$ and $\bc^2$ differ from $\bc^0$. Furthermore, if the partitions $P_{1,1}$, $P_{2,1}$, $P_{2,2}$, $
P_{2,3}$ and $P_{3,1}$ are as in Example~\ref{Ex_notation}, then $\Pa_{2,3} \cup \Pa_{3,1} = \supp(\bc^1)\cap \supp(\bc^2)$ and $A = p_{2,3} + p_{3,1}$.

\smallskip

In what follows, we briefly discuss a more general situation, where all the pairwise distances are not necessarily known. Let $\alpha$ (satisfying $0 \leq \alpha \leq \binom{s}{2}$) denote the number of known pairwise distances $d(\bc, \bc')$ and let $\mathcal{L'}$ denote the corresponding system of linear equations in terms of $p_{k,i}$. 
As the linear equations of $\mathcal{L}$ are linearly independent by Lemma~\ref{Lemma_lin_independence}, then so are the ones of $\mathcal{L'}$. Hence, the number of free variables in its solution is equal to 
\[
P_q(s) - 1 - \alpha \text.
\]
Therefore, besides the $\alpha$ pairwise distances, at least $P_q(s) - 1 - \alpha$ of the parameters $p_{k,i}$ have to be known. Moreover, a suitable set of $P_q(s) - 1 - \alpha$ free variables is enough to determine all $p_{k,i}$ (excluding $p_{1,1}$). Furthermore, if the values $p_{k,i}$ other than $p_{1,1}$ are known, then we have
\[
p_{1,1} = n -  \sum_{k=2}^{\min\{q,s\}} \sum_{i=1}^{\cS(s,k)} p_{k,i} \text.
\]
By considering the free variables as supplementary parameters alongside the $\alpha$ known pairwise distances, we can immediately derive the following theorem.
\begin{theorem}
	Let $S = \{\bc^1, \bc^2, \dots, \bc^s\}$ be a subset of $\Z_q^n$ and $T = (t_1, t_2, \ldots, t_s)$ a list of positive integers. If we know $\alpha$ of the pairwise distances $d(\bc^j, \bc^{j'})$ as well as a suitable set of $P_q(s) - 1 - \alpha$ values $p_{k,i}$, then all the values $p_{k,i}$ can be determined and the exact size of the intersection $I_T(S)$ follows by Theorem~\ref{Thm_intersection_size}.
\end{theorem}

In the following remark, we briefly discuss the asymptotic behavior of $P_q(s) - 1 - \binom{s}{2}$.
\begin{remark} \label{Rmk_asymptotics}
    In order to determine the asymptotic behavior of $P_q(s) - 1 - \binom{s}{2}$ as $s$ grows, we first notice that
    \[
    \cS(s,k) = \sum_{i=1}^k \frac{(-1)^{k-i}i^s}{(k-i)!i!} = \frac{k^s}{k!}\left(1 + \sum_{i=1}^{k-1} (-1)^{k-i} \binom{k}{i} \left(\frac{i}{k}\right)^s \right) \sim \frac{k^s}{k!} \in \Theta(k^s) \text,
    \]
     that is, $\cS(s,k)/(k^s/k!) \xrightarrow[]{s\to\infty} 1$, where $k$ is assumed to be a constant. Therefore, when $q$ is a constant and $s$ grows, we further obtain 
    \[
    P_q(s)=\sum_{k=1}^q\cS(s,k) \sim \sum_{k=1}^q \frac{k^s}{k!} \in \Theta(q^s) \text.
    \]
    Thus, as $\binom{s}{2}  \in \Theta(s^2)$, we have
    \[
    P_q(s) - 1 - \binom{s}{2} \in \Theta(q^s) \text.
    \]
    Hence, the requirement for additional parameters grows rapidly with respect to $s$. Recall by Theorem~\ref{Thm_intersection_size} that knowing all the $P_q(s) \in \Theta(q^s)$ values of $p_{k,i}$ is crucial in computing the cardinality of $I_t(S)$. Although the number of $p_{k,i}$ grows exponentially with respect to $s$, it should be noted that at most $n$ of the values can be non-zero due to the fact that $\sum_{k,i} p_{k,i}=n$. This also restricts the number of non-zero $j_{k,i}^l$'s in Theorem~\ref{Thm_intersection_size}, which is an important factor affecting the algorithmic complexity of the formula of Theorem~\ref{Thm_intersection_size} as is further discussed in Remark~\ref{RmkComplexity}.

    Assume then that $s$ is a constant and $q$ increases. While $q < s$, increasing $q$ also makes $P_q(s)$ larger; namely, for $2 < q < s$, we have $P_{q-1}(s) < P_q(s)$. However, if $q \geq s$, then $P_{q}(s) = P_s(s)$ stays constant. Therefore, we have $P_q(s)-1-\binom{s}{2}\in\Theta(1)$ for increasing $q$ and fixed $s$. Thus, increasing $q$ does not complicate the structural interplay of the centers as is the case with an increasing $s$.  
\end{remark}

The main focus of the section has been to study the theoretical and structural aspects of the intersection $I_T(S)$ and its cardinality. In the following remark, we show that although $P_q(s) \in \Theta(q^s)$, the size of the intersection can be computed in polynomial time (with respect to the size of the input $\Theta(n)$) using Theorem~\ref{Thm_intersection_size}, when the size $q$ of the alphabet and the number $s$ of the balls are bounded by constants. 
Note also that such assumptions are rather natural concerning applications such as DNA-storage. Moreover, we compare the obtained complexity to that of a `brute-force' style algorithm. However, a more thorough coverage of the complexity issues (such as the space complexity) is left for future research.

\begin{remark} \label{RmkComplexity}
    Let $q$ and $s$ be constants. To compute the formula of Theorem~\ref{Thm_intersection_size}, we first observe by~\eqref{Eq_Pqs} that  $P_q(s)$ is constant. Therefore, the number of parameters $j_{k,i}^l$ is upper bounded by a constant $qP_q(s)$. For efficiently determining the set $\mathcal{J}$ of tuples formed by $j_{k,i}^l$, we denote $t_{\min} = \min{T} \ (> 0)$ and the corresponding center by $\bc^{\min}$. Note that the length of any tuple in $\mathcal{J}$ is determined by the number of parameters $j_{k,i}^l$ and is at most $qP_q(s)$. As each $j_{k,i}^l\leq n$, any tuple in $\mathcal{J}$ can be presented using $\Ordo(\log_2 n)$ bits. Related to Theorem~\ref{Thm_intersection_size}, consider then the condition~(ii)  for $t_{\min}$ and, by substituting the condition~(i) $p_{k,i} = \sum_{l=1}^{\min\{q,k+1\}} j_{k,i}^l$ for all pairs of $k$ and $i$, obtain the inequality 
    \begin{equation} \label{EqRmkInequality}
        \sum_{k=1}^{\min\{q,s\}} \sum_{i=1}^{\cS(s,k)} \sum_{\substack{l=1\\ l\neq f_{k,i}(\bc^{\min})}}^{\min\{q,k+1\}}j_{k,i}^l \leq t_{\min} \text.
    \end{equation}
    Denote the tuples of $j_{k,i}^l$ satisfying this inequality by $\mathcal{J'}$. Observe that $\mathcal{J} \subseteq \mathcal{J'}$ since the values $f_{k,i}(\bc^{\min})$ (not involved in the inequality) are fixed by the condition~(i). By a rather standard `stars-and-bars' argument (for example, see \cite[Section~1.5]{levin2021discrete}), we obtain that
    \[
    |\mathcal{J}| \leq |\mathcal{J'}| = \binom{t_{\min}+\beta}{\beta} \leq \binom{t_{\min}+(q-1)P_q(s)}{(q-1)P_q(s)} \in \Ordo\left(t_{\min}^{(q-1)P_q(s)}\right)\text,
    \]
    where $\beta \ (> 0)$ stands for the number of parameters $j_{k,i}^l$ appearing in Inequality~\eqref{EqRmkInequality}. It is also rather immediate that the tuples of $\mathcal{J'}$ can be listed using $\Ordo\left(t_{\min}^{(q-1)P_q(s)} \cdot \log_2{n}\right)$ time since each tuple in $\mathcal{J}$ can be presented using $\Ordo(\log_2 n)$ bits. In order to determine which tuples of $\mathcal{J'}$ actually belong to $\mathcal{J}$, we need to check which of them satisfy the inequalities in (ii) for the other values of $t_m$. Notice first that the left-hand side of each inequality is at most $n$ since $\sum_{k=1}^{\min\{q,s\}} \sum_{i=1}^{\cS(s,k)} p_{k,i} = n$. Therefore, calculating the left-hand side requires approximately $P_q(s)$ additions and subtractions of integers with $\Ordo(\log_2{n})$ bits. Hence, due to the assumption $t_m \leq n$ (made in the introduction), the time complexity of checking each of the $s-1$ other inequalities is $\Ordo(\log_2{n})$. Thus, based on $\mathcal{J'}$, the set $\mathcal{J}$ can be computed in $\Ordo\left(t_{\min}^{(q-1)P_q(s)} \cdot \log_2{n}\right)$ time. In conclusion, the total time complexity of determining $\mathcal{J}$ is $\Ordo\left(t_{\min}^{(q-1)P_q(s)} \cdot \log_2{n}\right)$.

    After $\mathcal{J}$ has been determined, for \eqref{Eq2ndFormula} and \eqref{Eq1stFormula}, we need to calculate $P_q(s)$ multinomial coefficients and powers $(q-k)^{j_{k,i}^{k+1}}$ for each tuple of $\mathcal{J}$. Each multinomial coefficient 
    \[
    \binom{p_{k,i}}{j_{k,i}^1, j_{k,i}^2, \dots, j_{k,i}^{k+1}} = \frac{p_{k,i}!}{j_{k,i}^1! \cdot j_{k,i}^2!  \cdots  j_{k,i}^{k+1}!}
    \]    
    can be calculated in polynomial time with respect to $n$. Indeed, by~\cite{araujo2021multicoefficient}, the coefficient can be computed using $\Ordo(p_{k,i}) = \Ordo(n)$ multiplications (and divisions). Furthermore, each multiplication (and division) involves $\Ordo(n\log_2{n})$-bit integers since by Stirling's formula, $n!$ can be represented using $b = \log_2(n!) = n\log_2{n} - n\log_2{e} + \Ordo(\log_2{n}) = \Ordo(n\log_2{n})$ bits. 
    Each power $(q-k)^{j_{k,i}^{k+1}} \leq q^{j_{k,i}^{k+1}}$ can be computed using $\Ordo(\log_2{n})$ multiplications based on the exponentiation by squaring, where the multiplied integers have at most $\Ordo(n)$ bits (as $q^{j_{k,i}^{k+1}} \leq q^n = |\Z_q^n|$). 
    Thus, for each tuple of $\mathcal{J}$, the products in~\eqref{Eq2ndFormula} and \eqref{Eq1stFormula} can be computed using at most $2P_q(s)$ multiplications of integers, which can be represented using $\Ordo(\log_2{q^n}) = \Ordo(n)$ bits, and the sums using approximately $P_q(s)$ additions of $\Ordo(n)$-bit integers. In total, after $\mathcal{J}$ has been determined, we need
    \begin{itemize}
        \item $\Ordo\left(t_{\min}^{(q-1)P_q(s)}\right)$ additions/subtractions of $\Ordo(\log_2{n})$-bit integers and approximately $P_q(s)$ additions of $\Ordo(n)$-bit integers.  
        \item $\Ordo(n)$ multiplications of $\Ordo(n\log_2{n})$-bit integers, $\Ordo\left(\log_2{n}\right)$ multiplications of $\Ordo(n)$-bit integers and at most $2P_q(s)$ multiplications of $\Ordo(n)$-bit integers, of which the $\Ordo(n)$ multiplications of $\Ordo(n\log_2{n})$-bit integers is algorithmically the most demanding task.
    \end{itemize}
    By schoolbook algorithms, addition/subtraction of a $b$-bit integer can be performed in $\Ordo(b)$ time and multiplication has time complexity $\Ordo(b^2)$. However, we assume that a more efficient algorithm working in $\Ordo(b\log_2{b})$ is used for the multiplication (see~\cite{harvey2021multiplication}). 
    Thus, in conclusion, the time complexity of computing $|I_T(S)|$ is $\Ordo\left(t_{\min}^{(q-1)P_q(s)} \cdot n^2(\log_2{n})^2\right)$. Furthermore, as $t_{\min} \leq n$, it follows that the size of $I_T(S)$ can be computed in polynomial time $\Ordo\left(n^{2+(q-1)P_q(s)}(\log_2{n})^2\right)$ with respect to the length of the input $\Theta(n)$.

    Consider next the brute-force method, where we first compute the intersection $I_T(S) \subseteq \Z_q^n$ and then calculate its size. For this purpose, assume for simplicity that $t_{\min} < n/2$. In this approach, we first construct the ball $B_{t_{\min}}(\bc^{\min})$, which by the previous assumption is a smallest one, and then check whether its words belong to the other balls. The cardinality of the ball $B_{t_{m}}(\bc^m)$ satisfies
    \begin{equation} \label{Eq_Complexity}
    |B_{t_{\min}}(\bc^{\min})| 
    = \sum_{k=0}^{t_{\min}} \binom{n}{k}(q-1)^{k} \geq \binom{n}{t_{\min}}(q-1)^{t_{\min}} \geq \left( \frac{n}{t_{\min}} \right)^{t_{\min}}(q-1)^{t_{\min}}\text. 
    \end{equation}
    Therefore, the construction of $B_{t_{\min}}(\bc^{\min})$ takes at least $(n(q-1)/t_{\min})^{t_{\min}}$ bit operations. In what follows, we compare the complexity of computing $|I_T(S)|$ based on Theorem~\ref{Thm_intersection_size} and the brute-force method in a couple of natural cases, where $t_{\min}$ is constant or depends linearly on $n$:
    \begin{itemize}
        \item If $t_{\min}$ is a constant, then $\mathcal{J}$ can be computed in $\Ordo(\log_2{n})$ time and, hence, calculating the formula of Theorem~\ref{Thm_intersection_size} takes $\Ordo\left(n^2(\log{n})^2\right)$ time. In comparison, the time complexity of the brute-force method is at least $(n(q-1)/t_{\min})^{t_{\min}} \in \Omega(n^{t_{\min}})$. Thus, if $t_{\min} \geq 3$, computing $|I_T(S)|$ using the formula of Theorem~\ref{Thm_intersection_size} is faster than the brute-force method. (Note also that for $t_{\min} \in \{1,2\}$, the size of the intersection can be efficiently computed with both approaches, although depending on the implementation, the brute-force method might be faster.)

        \item 
        If $t_{\min}$ depends linearly on $n$, say $t_{\min}=\tau n$
        for some fixed positive real number $\tau<1/2$, then the complexity
        \[
        n\left( \frac{n}{t_{\min}} \right)^{t_{\min}}(q-1)^{t_{\min}} \geq \left( \frac{q-1}{\tau} \right)^{\tau n}
        \]
        of the brute-force method is exponential on $n$. In comparison, the computation of the formula can be performed in polynomial time $\Ordo\left((\tau n)^{(q-1)P_q(s)} \cdot n^2(\log{n})^2\right) = \Ordo\left(n^{2+ (q-1)P_q(s)}(\log{n})^2\right)$.
    \end{itemize}
    Thus, if $t_{\min} \geq 3$, then Theorem~\ref{Thm_intersection_size} leads to a faster algorithm than the brute-force method.
\end{remark}

\section{Asymptotic Results for the Intersections}\label{SecIntersection}

We begin the section by introducing some additional definitions and notations. 
For a set $S\subseteq \Z^n_q$ and a word $\bv\in \Z^n_q$, the smallest integer $r\geq0$ such that $S \subseteq B_r(\bv)$ is denoted by $d_M(\bv,S)$. In other words, we have $d_M(\bv,S)=\max\{d(\bv,\bs)\mid \bs\in S\}$. Observe that $\bv \in I_t(S)$ if and only if $d_M(\bv,S) \leq t$.
We further denote by $C_d(S)\subseteq \Z^n_q$ the set of words whose maximum distance to words in $S$ is exactly $d$, that is, $C_d(S)=\{\bv\in \Z^n_q\mid d_M(\bv,S)=d\}$. 
Furthermore, we call $d(S) = \min\{d_M(\bx, S) \mid \bx \in \F^n_q\}$ the \textit{radius of the set} $S$. Moreover, the set $C(S)$ of \emph{central words} consists of the words $\bv \in \Z^n_q$ such that $d(S) = d_M(\bv,S)$, that is, 
$$C(S)=\{\bv\in \Z_q^n\mid d_M(\bv,S)= d(S)\}.$$ 
In the following lemma, we determine 
(for sufficiently large $n$ and with other parameters being fixed) the main term of the cardinality $|I_t(S)|$. 

\begin{lemma}\label{lemAsymptoticDistCent}
Let $c,t$ and $q$ be fixed integers.
Given a nonempty set $S\subseteq \F^n_q$ with at most $c$ words 
such that $t \geq d(S)$, then $$|I_t(S)|=|C(S)|(q-1)^{t-d(S)}\binom{n}{t-d(S)} +E(n,t,d(S))$$ where $|E(n,t,d(S))|\in\mathcal{O}(n^{t-d(S)-1})$ and $|C(S)|\in \Theta(1)$ remains positive as $n$ increases.
%
\end{lemma}
\begin{proof}
Let us fix constants $c,t$ and $q$ and let $S=\{\bc^1,\dots,\bc^{c'}\}$, for $c'\leq c$, where $d(S)\leq t$. Due to $d(S)\leq t$, the intersection $I_t(S)$ is nonempty. Thus, as $t$ is a constant, the maximum distance $d$ between words in $S$ is also upper bounded by some constant; namely, by $2t$. 
Let $D \subseteq [1,n]$ denote the coordinate positions in which some of the words of $S$ differ (pairwise); in other words, all the words of $S$ agree on the coordinates $[1,n] \setminus D$. Since a pair of words in $S$ differs in at most $d$ coordinate positions and there are $\binom{c'}{2}$ word pairs, we have $|D| \leq d\binom{c'}{2}$. Thus, $|D|$ is bounded from above by a constant.
%
%

Let $\bv$ be an arbitrary word of $\Z^n_q$. We may write the word as a sum $\bv = \bv^1 + \bv^2$, where $\bv^1, \bv^2 \in \Z^n_q$ are words such that 
\begin{itemize}
    \item the coordinates of $\bv^1$ have the same value as all the words $\bc^j$ outside of $D$, that is, $v_i^1=c_i^j$ for each $i\in [1,n]\setminus D$ and $j\in[1,c']$, while the symbols within the coordinate positions in $D$ are unrestricted, and 
    \item each non-zero symbol of $\bv^2$ (if any) is outside the coordinates of $D$, that is, $\supp(\bv^2)\subseteq [1,n]\setminus D$.
\end{itemize}
Observe that we have $d_M(\bv,S)=d_M(\bv^1,S)+w(\bv^2)$. 
Hence, we have $\bv\in I_t(S)$ if and only if $w(\bv^2)\leq t-d_M(\bv^1,S)$. This further implies that $w(\bv^2)\leq t-d(S)$. In the following, we split our considerations into the cases with $w(\bv^2)=t-d(S)$ and $w(\bv^2)\leq t-d(S)-1$. We will  show that the main term $|C(S)|(q-1)^{t-d(S)}\binom{n}{t-d(S)}$ follows from the first case and that the error term $E(n,t,d(S))$ contains the cases with $w(\bv^2)\leq t-d(S)-1$ together with the possible overflow from the first case.

Let us first consider the case with the largest possible weight $w(\bv^2)=t-d(S)=t'$ of $\bv^2$. It is obtained if and only if
 we have $\bv^1\in C(S)$. Furthermore, there exist exactly \begin{equation*}
|C(S)|(q-1)^{t'}\binom{n-|D|}{t'}=|C(S)|(q-1)^{t'}\binom{n}{t'}\frac{(n-|D|)!(n-t')!}{n!(n-|D|-t')!}
\end{equation*}
 such words $\bv$. We may estimate 
$\frac{(n-|D|)!(n-t')!}{n!(n-|D|-t')!}\leq \frac{(n-|D|)^{t'}}{(n-t')^{t'}}=\left(\frac{n-|D|}{n-t'}\right)^{t'}=\left(1+\frac{t'-|D|}{n-t'}\right)^{t'}\leq\left(1+\frac{2t'-|D|}{n}\right)^{t'} \leq 1+t'\frac{2t'-|D|}{n}+\frac{p'}{n^2}\leq 1+\frac{p}{n}$
for some constants $p'$ and $p\geq0$ when $n$ is large enough. Indeed, the second inequality is due to the mediant inequality and the third inequality follows from the binomial expansion. Moreover, we have 
$\frac{(n-|D|)!(n-t')!}{n!(n-|D|-t')!}\geq \frac{(n-t'-|D|)^{t'}}{n^{t'}}=\left(1-\frac{t'+|D|}{n}\right)^{t'}\geq 1-t'\frac{t'+|D|}{n}\geq 1- \frac{b}{n}$ 
by Bernoulli's inequality for some constant $b\geq0$ and sufficiently large $n$  (ensuring that $(t'+ |D|)/n \leq 1$). Therefore, 
$$|C(S)|(q-1)^{t'}\binom{n}{t'}\left(1-\frac{b}{n}\right)\leq|C(S)|(q-1)^{t'}\binom{n-|D|}{t'}\leq |C(S)|(q-1)^{t'}\binom{n}{t'}\left(1+\frac{p}{n}\right)\text.$$ 
We note that since $\binom{n}{t'}\in \mathcal{O}(n^{t'})$, we have $\frac{b}{n}\binom{n}{t'}\in \mathcal{O}(n^{t'-1})$ and the same is true for $\frac{p}{n}\binom{n}{t'}$. Hence, both $\frac{b}{n}\binom{n}{t'}$ and $\frac{p}{n}\binom{n}{t'}$ are absorbed into the error term $E(n,t,d(S))$.

Let us then consider the words $\bv$ with $w(\bv^2)\leq t'-1$. There are at most $V_q(n,t'-1)$ such words $\bv$  in $I_t(S)$ for each $\bv^1$. Since there are at most $q^{d\binom{c'}{2}}$ ways to choose the word $\bv^1$, there are $\mathcal{O}(n^{t-d(S)-1})$ words $\bv$ with $w(\bv^2)\leq t'-1=t-d(S)-1$ in $I_t(S)$. Hence, in total, we have $|I_t(S)|=|C(S)|(q-1)^{t-d(S)}\binom{n}{t-d(S)}+E(n,t,d(S))$ where $|E(n,t,d(S))|\in\mathcal{O}(n^{t-d(S)-1})$. Finally, in order to determine the cardinality of $C(S)$, we observe that if $\bv\in C(S)$, then $\bv^2=\zero$ since $d_M(\bv,S)=d_M(\bv^1,S)+w(\bv^2)$ and $\bv = \bv^1$. Hence, as the number of ways to choose $\bv^1$ is bounded from above by a constant and since by definition $C(S)\neq\emptyset$, we have positive $|C(S)|\in \Theta(1)$.
%
%
%
\end{proof}

As a corollary of the previous result, we obtain the following lemma determining the asymptotic size of $I_t(S)$ (with respect to $n$).
\begin{lemma}\label{lemAsymptoticDist}
Let us fix integers $c,t$ and $q$.
Given a nonempty set $S\subseteq \F^n_q$ with at most $c$ words 
such that $t \geq d(S)$, then $|I_t(S)|\in \Theta(n^{t-d(S)})$ as $n$ increases.
%
\end{lemma}
\begin{proof}
    By Lemma~\ref{lemAsymptoticDistCent}, we have $|I_t(S)|=|C(S)|(q-1)^{t-d(S)}\binom{n}{t-d(S)} +E(n,t,d(S))$ where $|E(n,t,d(S))|\in\mathcal{O}(n^{t-d(S)-1})$ and $|C(S)|\in \Theta(1)$ is a positive integer bounded from above by a constant as $n$ increases. Since $t\geq d(S)$ and $t$ is a constant, we have  $\binom{n}{t-d(S)}\in \Theta(n^{t-d(S)})$. Hence, $|I_t(S)|=|C(S)|(q-1)^{t-d(S)}\binom{n}{t-d(S)} +E(n,t,d(S))\in \Theta\left(|C(S)|(q-1)^{t-d(S)}n^{t-d(S)}+E(n,t,d(S)\right)=\Theta\left(n^{t-d(S)}\right)$.
\end{proof}

When $t=d(S)$, we obtain the following simplification of Lemma~\ref{lemAsymptoticDistCent}. Observe that the following lemma allows us to compute the exact number of central words with Theorem~\ref{Thm_intersection_size}.\begin{lemma}\label{cort=dS}
Given a nonempty set $S\subseteq \F^n_q$ 
such that $t = d(S)$, then $$I_t(S)=C(S)\text.$$
\end{lemma}
\begin{proof}
Let $\bv\in I_t(S)$. Then we have $d_M(\bv,S)\leq t$. Hence, $I_t(S)\subseteq C(S)$. Moreover, we have $C(S)\subseteq I_t(S)$ since each vertex in $C(S)$ has distance of at most $d(S)=t$ to words in $S$.
\end{proof}

As Lemma \ref{lemAsymptoticDist} suggests, given a sufficiently large $n$, to find a set of words $S\subseteq \F_q^n$ for which the intersection $|I_t(S)|$ achieves its maximum value, we are interested in finding the radius $d(S)$ of the set $S$ and the central words $C(S)$. Notice that in the following lemma, if $|S|=1$, then we consider the pairwise distance of words in $S$ to be $0$.

\begin{lemma}\label{lemMindS}
Let us fix integers $c$ and $d$ and consider a nonempty set $S\subseteq \F^n_q$ with at most $c$ words with pairwise distances exactly equal to $d$.
\begin{enumerate}
\item If $n\geq d$, then $d(S)\geq\lceil d/2\rceil$. 

\item If $n \geq 1 + c \lfloor d/2 \rfloor$ and we have $q\geq c$ or $d$ is even, then there exists a set $S$ attaining the  bound in 1).
\end{enumerate}
\end{lemma}
\begin{proof}
Notice that if $|S|=1$, then $d=0$ and the claim follows for both statements. Let $\bc_1,\bc_2\in S$ and $\by\in C(S)$. Due to the triangular inequality, we have $d=d(\bc_1,\bc_2)\leq d(\bc_1,\by)+d(\bc_2,\by)\leq 2d_M(\by,S)=2d(S)$. Therefore, $d(S)\geq d/2$.

 Let us next consider point 2) and construct $S$ so that we attain the lower bound $\lceil d/2\rceil$. Let us first consider the case with even $d\geq2$ (hence, $c\geq2$). Let $S=\{\bc^1,\dots,\bc^c\}$ where $c_j^i\in\{0,1\}$ for each $i,j$ and 
 $\supp(\bc^i)=[(i-1)\cdot d/2+1,i\cdot d/2]$
 for each $i\in[1,c]$. Notice that the pairwise distance of all words in $S$ is exactly $d$. Moreover, as $i\cdot \frac{d}{2}\leq 1+ c\cdot \frac{d}{2}\leq n$, the construction fits in $\Z^n_q$. Consider $\bw=\zero\in \Z^n_q$. 
 Now, we have $d_M(\bw,S)=d/2$.


Let us next consider the case with odd $d$ and $q\geq c$. Let $S=\{\bc^1,\dots,\bc^c\}$ where $c_j^i\in\{0,1\}$ for each $i\in [1,c]$ and $j\in[1,n-1]$, $c_n^i=i-1$, and $\supp(\bc^i)\setminus\{n\}=\left[(i-1)\cdot \lfloor d/2\rfloor+1,i\cdot \lfloor d/2\rfloor\right]$ for each $i\in[1,c]$. Note that since $n \geq 1 + c \lfloor d/2 \rfloor$, we can make the above choices for the supports. Notice that the pairwise distance of all words in $S$ is $2\lfloor d/2\rfloor+1=d$. For $\zero\in\Z^n_q$, we have $d_M(\zero,S)=\lfloor d/2\rfloor+1=\lceil d/2\rceil$ as $c\geq2$.
\end{proof}

\subsection{Intersection of Three Substitution Balls}\label{sec3balls}
In this section, we focus on the cardinality of an intersection of three $q$-ary substitution (Hamming)  balls with $q\geq3$ centered at words $\bc^0,\bc^1,\bc^2$ with pairwise distances $d=d(\bc^0,\bc^1)=d(\bc^0,\bc^2)=d(\bc^1,\bc^2)$ (recall that the case of three balls with $q=2$ is known); in particular, we focus on determining the maximum cardinality of the intersection. We first concentrate on the equidistant case as we later manage to show that it leads to the maximum size intersection in Theorem~\ref{thmMax3IntExactNED}. Note that as we are only interested in the size of the intersections, by applying the symmetries of the Hamming space, we may assume, without loss of generality, that $\bc^0=\zero$, $\bc^1=1^d0^{n-d}$ and $c_j^2\in\{0,1,2\}$ for each $j$. These assumptions can be made by first subtracting $\bc_0$ from each word, then observing that, from the point of view of distance, the ordering of coordinate positions does not matter and finally by permuting the symbols within each coordinate position. We further denote by $S_d^A=\{\bc^0,\bc^1,\bc^2\}$ a triple of words such that $d(\bc^i,\bc^j)=d$ for each $i\neq j$ and $|\supp(\bc^1)\cap \supp(\bc^2)|=A$.
Recall that in Example~\ref{exSubInt}, we have $A = 1$ in the case of the set $S$, while $A = 2$ in the case of $S'$.
Note that $A\geq d/2$. Indeed, otherwise we would have $d(\bc^1,\bc^2)>d$.

In the following lemma, we show that $d$ and $A$ determine the size of the intersection of three Hamming balls of radius $t$ each at distance $d$ from each other.

\begin{lemma}\label{lem3intSizesBelong}
Given $t\geq d(S_d^A)$, $d<n$ and $q\geq3$, the size of intersection of three $t$-radius Hamming balls in $\Z^n_q$ with centers at  distances $d$ from each other is within the set $$\left\{|I_t(S_d^A)|\mid A\in[\lceil d/2\rceil, d]\right\}\text.$$
\end{lemma}
\begin{proof}
Consider a triple of codewords $\bc^0,\bc^1,\bc^2$ at pairwise distance $d$ from each other. As we have discussed at the beginning of this section, when we consider the cardinality of $|B_t(\bc^0)\cap B_t(\bc^1)\cap B_t(\bc^2)|$, we may assume without loss of generality that $\bc^0=\zero$, $w(\bc^1)=w(\bc^2)=d$, $c_i^1=1$ for $i\in[1,d]$ and $c_j^2\in\{0,1,2\}$ for each $j$. Let us denote $|\supp(\bc^1)\cap \supp(\bc^2)|$ by $A$. Consequently, this implies that $|B_t(\bc^0)\cap B_t(\bc^1)\cap B_t(\bc^2)|= |I_t(S_d^A)|$. Hence, we only need to show that $A\in[\lceil d/2\rceil,d].$ 

We have $A\leq d$ since $w(\bc^1)=d$. Thus, $d=d(\bc^1,\bc^2)\geq w(\bc^1)+w(\bc^2)-2A=2d-2A$ which implies that $A\geq d/2$ and since $A$ is an integer, we have $A\geq \lceil d/2\rceil$ as claimed.
\end{proof}

Since $\bc^0=\zero$, we have $w(\bc^1)=w(\bc^2)=d$, $|\supp(\bc^1)\cap \supp(\bc^2)|=A$ and $|\supp(\bc^1)\setminus\supp(\bc^2)|=|\supp(\bc^2)\setminus\supp(\bc^1)|=d-A\leq\lfloor d/2\rfloor$. Furthermore, let us denote by $A_1$ the number of coordinates in $\supp(\bc^1)\cap \supp(\bc^2)$ in which the symbols of $\bc^1$ and $\bc^2$ agree and by $A_2$ the number of coordinates in $\supp(\bc^1)\cap \supp(\bc^2)$ in which the symbols of $\bc^1$ and $\bc^2$ disagree. Thus, $A_1+A_2=A$. Notice that since $d=d(\bc^1,\bc^2)$, we have $d=2(d-A)+A_2$. Hence, $A_2=2A-d$ and $A_1=A-A_2=d-A$.

In the following lemma, we give exact value for the radius of the set $S_d^A$, where the pairwise distances of its words are equal to $d$.

\begin{lemma}\label{lem3centDist}
Given $d\leq n$, $A\in[\lceil d/2\rceil,d]$ and $q\geq3$, we have $$d(S_d^A)=\left\lceil\frac{d+A}{3}\right\rceil\text.$$
\end{lemma}
\begin{proof}
Let $\bw \in C(S_d^A)$ be a central word of $S_d^A$ such that firstly the sum $d_\bw=d(\bw, \bc^0) + d(\bw, \bc^1) + d(\bw, \bc^2)$ is minimum among the words in $C(S_d^A)$ and secondly, among the central words attaining the minimum sum $d_\bw$, the distance $d(\bw, \bc^0)$ is minimal.
Consider the distance between the word $\bw$ and the words $\bc^0,\bc^1$ and $\bc^2$. 
Let us denote $a_1= |\vsupp(\bw)\cap\vsupp(\bc^1)\cap \vsupp(\bc^2)|$. Note that $a_1\leq A_1$. Denote by $a_2^1$ (resp. $a_2^2$) the number of symbols of $\bw$ which agree only with $\bc^1$ (resp. $\bc^2$) but not with $\bc^2$ (resp. $\bc^1$) in $\supp(\bc^1)\cap \supp(\bc^2)$. Note that $a_2^1+a_2^2\leq A_2$. Furthermore, we denote by $b_1$ (resp. $b_2)$ the number of symbols of $\bw$ which agree with exactly $\bc^1$ (resp. $\bc^2$) in $\supp(\bc^1)\setminus \supp(\bc^2)$ (in $\supp(\bc^2)\setminus \supp(\bc^1)$). By the choice of $\bw$, we have $w(\bw)=a_1+a_2^1+a_2^2+b_1+b_2$. Indeed, any other non-zero symbol increases all of the distances to $\bc^0$, $\bc^1$ and $\bc^2$, and thus, $\bw$ would not be a central word. 
Hence, we may observe  the following  \begin{align}\label{eq1d}
d_0=d(\bw,\bc^0)&=a_1+a_2^1+a_2^2+b_1+b_2\text;\nonumber\\
d_1=d(\bw,\bc^1)&=d-a_1-a_2^1-b_1+b_2\text;\\
d_2=d(\bw,\bc^2)&=d-a_1-a_2^2+b_1-b_2\text.\nonumber
\end{align}
Since $b_1$ and $b_2$ are in symmetrical positions in Equations (\ref{eq1d}), we may assume without loss of generality that $b_1\geq b_2$. As  $\bw$ is a central word, it minimizes the maximum distance to $S^A_d$, that is, minimizes the value $\max\{d_0,d_1,d_2\}$. Hence, we may assume that $b_1=B$ and $b_2=0$. Indeed, if $b_1=B+B'$ and $b_2=B'$, then decreasing both values by $B'$ would not change the values $d_1$ and $d_2$ while it would decrease $d_0$ by Equations (\ref{eq1d}). This is a contradiction to $\bw$ minimizing the sum $d_\bw=d_0+d_1+d_2$. This justifies the assumption $b_2=0$.

Next, we show  $d_1=d_2$. Suppose, to the contrary, that $d_1 \neq d_2$.

\textbf{Case 1:} Assume first that 
$d_2-d_1>0$. This implies $a_2^1+2b_1-a^2_2>0$. Hence, $a_2^1>0$ or $b_1>0$. Consider a word $\bw'$ obtained from $\bw$ by changing a non-zero symbol contributing to either $b_1$ or $a_2^1$ to zero. If we decrease $a_2^1$, then we change a non-zero symbol of $\bw$ which agrees with $\bc^1$ but not with $\bc^2$ in $\supp(\bc^1)\cap \supp(\bc^2)$ into a zero. If we decrease $b_1$, then we change a non-zero symbol of $\bw$ which agrees with $\bc^1$ in $\supp(\bc^1)\setminus \supp(\bc^2)$ into a zero. Note that $d(\bw',\bc^0)= d_0-1$, $d(\bw',\bc^1)= d_1+1$ and  $d(\bw',\bc^2) \leq d_2$. Hence, $d(\bw',\bc^0)+d(\bw',\bc^1)+d(\bw',\bc^2)\leq d_\bw$. Moreover, $d_2>d_1$ implies that $d_M(\bw',S_d^A)\leq \max\{d_0-1,d_2\} \leq d_M(\bw,S_d^A)$. Thus, $\bw'$ is a central word. 
However, since $d(\bw',\bc^0)<d(\bw,\bc^0)$, we have a contradiction with the minimality of $d_0$.

\textbf{Case 2:} Suppose then that 
 $d_1-d_2>0$. Now $a^2_2-a_2^1-2b_1>0$ and, hence, $a_2^2>0$. Consider a word $\bw'$ obtained from $\bw$ by changing a non-zero symbol contributing to $a_2^2$ to zero, that is, we change a non-zero symbol of $\bw$ which agrees with $\bc^2$ but not with $\bc^1$ in $\supp(\bc^1)\cap \supp(\bc^2)$ into a zero. Again, we note that $d(\bw',\bc^0)= d_0-1$, $d(\bw',\bc^1)= d_1$ and  $d(\bw',\bc^2) = d_2 + 1$. Thus, $d(\bw',\bc^0)+d(\bw',\bc^1)+d(\bw',\bc^2)=d_\bw$. Moreover,  $d_1>d_2$ implies $d_M(\bw',S_d^A)\leq \max\{d_0-1,d_1\} \leq d_M(\bw,S_d^A)$. Thus, $\bw'$ is a central word. 
 However, since $d(\bw',\bc^0)<d(\bw,\bc^0)$, we have a contradiction with the minimality of $d_0$.\smallskip

Hence, we may assume from now on that $d_1=d_2$. Together with $b_2=0$, this implies that $a_2^2=a_2^1+2b_1$. 

We continue by showing that $b_1=0$. Suppose, to the contrary, that $b_1=B>0$.  Since $a_2^2=a_2^1+2b_1$, we also get $a_2^2>B$. Consider a word $\bw'$ obtained from $\bw$ by changing $B$ non-zero symbols of $\bw$ contributing to $b_1$ into zeros and $B$ non-zero symbols of $\bw$ contributing to $a_2^2$ into symbols contributing to $a_2^1$ (that is, we decrease both $b_1$ and $a_2^2$ by $B$ and increase $a_2^1$ by $B$). To be more precise, these changes are done by changing $B$ non-zero symbols in $\supp(\bc^1)\setminus \supp(\bc^2)$ into zeros and by changing $B$ non-zero symbols in $\supp(\bc^1)\cap \supp(\bc^2)$ which agree with only $\bc^2$ into symbols which agree with only $\bc^1$. Now, $d(\bw', \bc^1)=d_1$ and $d(\bw', \bc^2)=d_2$ but $d(\bw', \bc^0)=d_0-B$. Again, we obtain a contradiction to the minimality of $d_{\bw}$. 
%
%
Hence, we have $b_1=0$ and thus, $a_2^2=a_2^1$.
%
This implies \begin{equation}\label{eq2d1}
d_1=d-a_1-a_2^1=d_2
\end{equation} and 
\begin{equation}\label{eq2d0}
d_0=a_1+2a_2^1\text.
\end{equation}


Note that $a_1\leq A_1=d-A\leq \lfloor d/2\rfloor$. Next, we show that $a_1=A_1$. If $A_1=0$, then we are done. Hence, assume that $A_1>0$. Suppose, to the contrary, that $a_1=A_1-A'$ for some positive integer $A'\leq A_1$. We have $d_1=d-a_1-a_2^1=A+A'-a^1_2\geq \lceil d/2\rceil +A'-a^1_2$ by Equation (\ref{eq2d1}) and $d_0=A_1-A'+2a_2^1\leq \lfloor d/2\rfloor -A'+2a^1_2$  by Equation (\ref{eq2d0}). We first show that $a_2^1 \geq 1$. For this purpose, suppose to the contrary that $a_2^1 = 0$. In this case, we have $d_0\leq \lfloor d/2\rfloor -A'\leq \lfloor d/2\rfloor -1$ and $d_1=d_2\geq \lceil d/2\rceil +1$. However, now we can obtain word $\bw'$ from $\bw$ by decreasing $A'$ by one. This change is done by changing one zero symbol of $\bw$ in $\vsupp(\bc^1)\cap\vsupp(\bc^2)$ into a symbol agreeing with both $\bc^1$ and $\bc^2$ which is possible since $A'>0$. The resulting word $\bw'$ has $d_M(\bw',S^A_d)=d_2-1$, a contradiction to $\bw$ being a central word. Hence, we may assume that $a_2^1\geq1$ (and $A'>0$).





Consider a word $\bw'$ obtained from $\bw$ by decreasing both $A'$ and $a^1_2$ by 1, which can be done as both are positive integers. The decrease of $A'$ is done as in the previous paragraph, while the decrease of $a^1_2$ is done by changing one non-zero symbol in $\supp(\bc^1)\cap\supp(\bc^2)$ agreeing with exactly $\bc^1$ into a zero. This change decreases the distance to $\bc^2$ while the other two distances remain equal. Thus, $\bw'$ is also a central word but has a smaller sum of distances to the words in $S^A_d$, contradicting the minimality of the sum $d_0+d_1+d_2$. Hence, we have $A'=0$ and $a_1=A_1$.
%
%
%
Thus, we obtain \begin{equation}\label{eq3d1}
d_1=d_2=d-A_1-a_2^1=A-a_2^1
\end{equation} and \begin{equation}\label{eq3d0}
d_0=A_1+2a_2^1=d-A+2a_2^1\text.
\end{equation} 

We may consider $d_0$ and $d_1$ as linear functions with parameter $a_2^1$, that is, as an increasing and as a decreasing line, respectively. 
If we had $a_2^1\in \R$, 
then the minimum value of the maximums would be obtained when the lines cross each other. However, as $a_2^1$ is an integer, we are interested in comparing the first integer value below and above the intersection point.
Consider  first when $d_0=d_1$. By Equations (\ref{eq3d1}) and (\ref{eq3d0}), this requires that $A_1+2a_2^1=d-A_1-a_2^1$, that is, $3a_2^1=d-2A_1=2A-d$ and $a_2^1=\frac{2A-d}{3}$. 
When $2A-d$ is divisible by 3, we have for $a_2^1=\frac{2A-d}{3}$ that $d_0=d_1=d_2=\frac{d+A}{3}$ by Equations~(\ref{eq3d1}) and~(\ref{eq3d0}). Since increasing or decreasing $a_2^1$ would increase $d_0$ or $d_1$, the maximum of the three distances would also increase. Therefore, in this case, we obtain $d(S_d^A)=\frac{d+A}{3}$.

Let us next consider the case when $2A-d$ is not divisible by three. In this case, the three distances cannot be equal. 
We consider the cases where $a_2^1\in\left\{\lfloor \frac{2A-d}{3}\rfloor, \lceil \frac{2A-d}{3}\rceil\right\}$. Consider first the case $2A-d=3h+1$ for some integer $h$. We have $a_2^1\in\left\{ \frac{2A-d+2}{3},\frac{2A-d-1}{3}\right\}$. If $a_2^1= \frac{2A-d+2}{3}$, then $d_0=d-A+2\frac{2A-d+2}{3}=\frac{d+A+4}{3}$ by Equation (\ref{eq3d0}) and
$d_1=d_2=A-\frac{2A-d+2}{3}=\frac{d+A-2}{3}$ by Equation (\ref{eq3d1}). If $a_2^1= \frac{2A-d-1}{3}$, then $d_0=d-A+2\frac{2A-d-1}{3}=\frac{d+A-2}{3}$ by Equation (\ref{eq3d0}) and
$d_1=d_2=A-\frac{2A-d-1}{3}=\frac{d+A+1}{3}$ by Equation (\ref{eq3d1}). Notice that the latter case obtains a smaller maximum value. Therefore, in this case, we obtain $d(S_d^A)=\frac{d+A+1}{3}=\lceil\frac{d+A}{3}\rceil$ when $a_2^1= \frac{2A-d-1}{3}$.

Consider then the case with $2A-d=3h+2$. We have $a_2^1\in\left\{ \frac{2A-d+1}{3},\frac{2A-d-2}{3}\right\}$. If $a_2^1= \frac{2A-d+1}{3}$, then $d_0=d-A+2\frac{2A-d+1}{3}=\frac{d+A+2}{3}$ by Equation (\ref{eq3d0}) and
$d_1=d_2=A-\frac{2A-d+1}{3}=\frac{d+A-1}{3}$ by Equation (\ref{eq3d1}). If $a_2^1= \frac{2A-d-2}{3}$, then $d_0=d-A+2\frac{2A-d-2}{3}=\frac{d+A-4}{3}$ by Equation (\ref{eq3d0}) and
$d_1=d_2=A-\frac{2A-d-2}{3}=\frac{d+A+2}{3}$ by Equation (\ref{eq3d1}). Notice that both cases obtain the same maximum value $\frac{d+A+2}{3}$. Hence, we obtain $d(S_d^A)=\frac{d+A+2}{3}=\lceil\frac{d+A}{3}\rceil$. To conclude, we have $d(S_d^A)=\lceil\frac{d+A}{3}\rceil$.
%
%
%
%
\end{proof}

Due to the previous results, we obtain the following lemma which gives the asymptotic size of the intersection based on the parameter $A$.
\begin{lemma}\label{lem3intSizesAsymp}
Given fixed integers $t,d,A$ and $q$ such that $t\geq d(S_d^A)$, $d<n$, $A\in[\lceil d/2\rceil,d]$ and $q\geq3$, 
we obtain with respect to $n$ that
$$\left|I_t(S_d^A)\right|\in \Theta\left(n^{t-\lceil\frac{d+A}{3}\rceil}\right)\text.$$
\end{lemma}
\begin{proof}
By Lemma \ref{lemAsymptoticDist} we have $|I_t(S_d^A)|\in \Theta(n^{t-d(S_d^A)})$. Hence, the claim follows from Lemma~\ref{lem3centDist}.
\end{proof}

As we have seen in the previous lemma, given a $q$-ary $e$-error-correcting code with minimum distance $d=2e+2$ and a codeword triple at pairwise equal distances $d$ with $A=\frac{d}{2}$, we obtain $|I_{e+1}(S_d^A)|\in \Theta(n^{e+1-\lceil\frac{3e+3}{3}\rceil})=\Theta(1)$. However, if we forbid triples with $A=\frac{d}{2}$, then we would have $A>\frac{d}{2}$ and $d(S^A_d)>e+1=t$. As we see in the following remark, this gives list-error-correcting codes, i.e., such error-correcting codes that (instead of a single codeword) the decoder gives a list of codewords including the transmitted one (for further detail see~\cite{Guruswamin_kirja}), since $I_{e+1}(S_d^A)=\emptyset$. 
\begin{remark}\label{RemList}
Let $S$ be an $e$-error-correcting code with minimum distance $d=2e+2$ with the following additional property: For each triple $\ba,\bb,\bc$ with $d(\ba,\bb)=d(\ba,\bc)=d(\bb,\bc)=d$ we never have $|\vsupp(\bb-\ba)\cap \vsupp(\bc-\ba)|=\frac{d}{2}$. From the perspective of Lemma \ref{lem3intSizesAsymp}, this implies that we forbid minimal $A=\frac{d}{2}$ (or in other words maximal $A_1$) for codeword triples with pairwise distances exactly $d$. 

Consider a ball of radius $e+1$ containing the codewords $\ba$, $\bb$ and $\bc$. We may assume that the ball is centered around the $\zero$-word. Due to the minimum distance $2e+2$, we have $w(\ba)=w(\bb)=w(\bc)=e+1$ and the supports of these words cannot intersect. However, now $|\vsupp(\bb-\ba)\cap \vsupp(\bc-\ba)|=\frac{d}{2}$, a contradiction. Thus,  $S$ is an $e$-error-correcting code with the  property $|B_{e+1}(\bw)\cap S|\leq 2$ for each $\bw\in \Z^n_q$. In other words, if obtain a word $\bw$ from $\bx\in S$ with at most $e+1$ errors, then we can decode $\bw$ into a list of length $2$ containing $\bx$.
%
\end{remark}

\subsubsection{Configurations maximizing the intersection of three balls}\label{SubsubSecMaxInt}

In this subsection, we consider how to obtain the \emph{maximum} size of the intersection of three Hamming balls with pairwise distances \emph{at least} $d$ from each other. 
In Lemma~\ref{lemMax3Int} and Theorem~\ref{thmMax3IntExact}, we continue with the equidistant case studied above and determine for which values of $A$ the asymptotic maximum of $|I_t(S_d^A)|$ is reached. Then, in Theorem~\ref{thmMax3IntExactNED}, we extend the results to the cases, where the pairwise distances of the centers are at least $d$ -- for example, such cases occur if the centers are chosen from a code with minimum distance $d$.

\begin{lemma}\label{lemMax3Int}
Given fixed integers $q,t$ and $d$ such that $q\geq3$, $t\geq1$ and $d\geq2$, the maximum cardinality of intersection, that is, $$\left|I_t(S_d^A)\right|\in \Theta\left(n^{t-\lceil d/2\rceil}\right)\text,$$ is asymptotically obtained over sets $S_d^A$ with $d(S^A_d)\leq t$ for an increasing $n$ if and only if $A=d/2$ for an even $d$ or $A\in\left\{\lceil d/2\rceil,\lceil d/2\rceil+1\right\}$ for an odd $d$.
\end{lemma}
\begin{proof}
By Lemma \ref{lem3intSizesAsymp}, we have $|I_t(S_d^A)|\in \Theta\left(n^{t-\lceil (d+A)/3\rceil}\right)$ for $A\in [\lceil d/2\rceil,d]$.
Observe that $t - \lceil (d+A) / 3 \rceil$ is a decreasing function for the argument $A$. Consider first an even $d$. We have for $A=d/2$, $t-\lceil (d+A)/3\rceil=t-d/2$ and for $A=d/2+1$, we have $t-\lceil (d+A)/3\rceil=t-\lceil d/2+1/3\rceil=t-d/2-1<t-d/2$. Hence, for an even $d$, we have $| I_t(S_d^A)|\in\Theta(n^{t-\lceil d/2\rceil})$ if and only if $A=d/2$ as claimed. 

Let us next consider odd $d\geq3$ and  $A=(d+1)/2$. We have $t-\lceil (d+(d+1)/2)/3\rceil=t-\lceil d/2+ 1/6\rceil=t-(d+1)/2=t - \lceil d/2 \rceil$. Similarly for $A=(d+3)/2$, we have $t-\lceil (d+(d+3)/2)/3\rceil=t-\lceil d/2 +1/2\rceil=t-(d+1)/2=t - \lceil d/2 \rceil$. Moreover, for $A=(d+5)/2$ and $d\geq5$,  we have $t-\lceil (d+(d+5)/2)/3\rceil=t-\lceil d/2 +1/2+1/3\rceil=t-(d+3)/2=t - \lceil d/2 \rceil - 1$. Hence, the claim follows.
%
%
%
%
\end{proof}

For $\bw\in \Z^n_q$, we denote by $\bw_{|n'}$ the restriction of $\bw$ to $\Z^{n'}_q$, that is, for $\bw'=\bw_{|n'}$ we have the word $\bw'\in\Z^{n'}_q$ and $w_i=w'_i$ for $1\leq i\leq n'\leq n$. We extend this notation also for sets of words in an analogous way, i.e., we denote $S_{|n'}$ for a set $S\subseteq \Z^n_q$. In the following theorem, we give the exact conditions for the equidistant center points of three balls to attain the maximum size intersection.

\begin{theorem}\label{thmMax3IntExact}
Given fixed integers $q,t$ and $d$ such that $q\geq3$, $t\geq 1$, $d\geq2$ and $n\geq n_0(q,t,d)$ for some sufficiently large threshold $n_0(q,t,d)$, the  value $|I_t(S_d^A)|$ obtains its maximum over sets $S_d^A$ with $d(S_d^A)\leq t$, when $$A=\begin{cases}
d/2,&\text{ if } d \text{ is even;}\\
(d+1)/2,&\text{ if } d \text{ is odd and } q>6\text;\\
(d+1)/2+1,&\text{ if } d \text{ is odd and } q<6\text.
\end{cases}
$$
\end{theorem}
\begin{proof}
Assume first that $d$ is even. By Lemma \ref{lem3intSizesAsymp}, the maximum value of $|I_t(S_d^A)|$ is in $\Theta(n^{t-\lceil d/2\rceil})$ given that $n$ is large enough. Furthermore, by Lemma \ref{lemMax3Int}, we have $|I_t(S_d^A)|\in \Theta(n^{t-\lceil d/2\rceil})$ if and only if $A=d/2$. Hence, the claim follows for an even $d$. Furthermore, for an odd $d$, we have $|I_t(S_d^A)|\in \Theta(n^{t-\lceil d/2\rceil})$, if and only if $A\in\{\lceil d/2\rceil,\lceil d/2\rceil+1\}$.

Hence, in the following, we concentrate on comparing $|I_t(S_d^A)|$ for an odd $d$ and values $A\in\{\lceil d/2\rceil,\lceil d/2\rceil+1\}$. Consider first the case $|I_t(S_d^{\lceil d/2\rceil})|$ with $S_d^{\lceil d/2\rceil}=\{\bc^0,\bc^1,\bc^2\}$. Now, we have $A_1 = d - A = \lceil d/2 \rceil - 1$ and $A_2 = 2A - d = 1$. Recall that $\bc^0=\zero$, $c_i^1=1$ exactly for $i\in[1,d]$ and $w(\bc^2)=d$. We may further assume that $c^2_d=2$ and $c^2_i=1$ exactly for $i\in[\lceil d/2\rceil, d-1]\cup [d+1,d+\lfloor d/2\rfloor]$. Otherwise, we will have $c_i^j=0$ for each $j$ and each other $i$.

Consider next a word $\bw_{|d+\lfloor d/2\rfloor}\in \Z^{d+\lfloor d/2\rfloor}_q$ and $\bw\in \Z^n_q$. By Lemma \ref{lemMindS}, we have $d_M(\bw_{|d+\lfloor d/2\rfloor},S^A_{d_{|d+\lfloor d/2\rfloor}})$ $\geq\lceil d/2\rceil$. Notice that if $d_M(\bw_{|d+\lfloor d/2\rfloor},S^A_{d_{|d+\lfloor d/2\rfloor}})=\lceil d/2\rceil+d'$ for some integer $d'\geq0$, then we may choose the rest of the symbols in $\bw$  in $[d+\lfloor d/2\rfloor+1,n]$ in $\Theta(n^{t-\lceil d/2\rceil-d'})$ ways. 
Hence, we will in the following count $|C_{\lceil d/2\rceil}(S^A_{d_{|d+\lfloor d/2\rfloor}})|$. 
Let us assume that $\bv\in C_{\lceil d/2\rceil}(S^A_{d_{|d+\lfloor d/2\rfloor}})=I_{\lceil d/2\rceil}(S^A_{d_{|d+\lfloor d/2\rfloor}})$. 

Let us first show that $v_i=0$ for each $i\in [1,\lfloor d/2\rfloor]\cup[d+1,d+\lfloor d/2\rfloor]$.
 Suppose to the contrary that $v_i\neq 0$ for some $i\in [1,\lfloor d/2\rfloor]\cup[d+1,d+\lfloor d/2\rfloor]$.
Assume that $i\in [1,\lfloor d/2\rfloor]$. The case with $i\in [d+1,d+\lfloor d/2\rfloor]$ is symmetric. Since $\lceil d/2\rceil\geq d(\bv, \bc^2)\geq 1+d-|\supp(\bv)\cap \supp(\bc^2)|$, we have $w(\bv)\geq 1+ |\supp(\bv)\cap \supp(\bc^2)|\geq 1+\lceil d/2\rceil$. However, now $d(\bv,\bc_0)\geq 1+\lceil d/2\rceil$, a contradiction. Hence, $\supp(\bv)\subseteq \supp(\bc^1)\cap\supp(\bc^2)=[\lceil d/2\rceil,d]$. 

As $d(\bv, \bc^j) \leq \lceil d/2 \rceil$ for both $j\in\{1,2\}$, we have $\min\{|\vsupp(\bv)\cap\vsupp(\bc^1)|,|\vsupp(\bv)\cap\vsupp(\bc^2)|\}\geq \lfloor d/2\rfloor$. Thus, $v_i=1$ for each $i\in [\lceil d/2\rceil, d-1]$. As $w(\bv)\leq \lceil d/2\rceil$, we can have at most one additional non-zero symbol and its only possible location is at coordinate $d$. Note that changing this symbol from $0$ to some other, does not increase the distances  $d(\bv, \bc^j)$ for either $j\in\{1,2\}$. This leads to the possibility that there are exactly $q$ different choices for $\bv$  and \begin{equation}\label{eq|C|d}
\left|C_{\lceil d/2\rceil}\left(S^A_{d_{|d+\lfloor d/2\rfloor}}\right)\right|=\left|C\left(S^{\lceil d/2\rceil}_{d}\right)\right|=q\text.
\end{equation}



Consider next the case with an odd $d$ and values $A=\lceil d/2\rceil+1$. Now, we have $A_1 = \lceil d/2 \rceil - 2$ and $A_2 = 3$. Again, assume that $\bc^0=\zero$, $w(\bc^1)=d$, $c_i^1=1$ exactly for $i\in[1,d]$, $w(\bc^2)=d$, $c^2_d=c^2_{d-1}=c^2_{d-2}=2$ and $c^2_i=1$ exactly for $i\in[\lfloor d/2\rfloor,d-3]\cup[d+1,d+\lfloor d/2\rfloor-1]$. As in the previous case, due to the asymptotics, we are interested in finding: How many words are there in $|C_{\lceil d/2\rceil}(S^A_{d_{|d+\lfloor d/2\rfloor-1}})|$? 

Let $\bv\in C_{\lceil d/2\rceil}(S^A_{d_{|d+\lfloor d/2\rfloor-1}})$. 
Let us again show that $v_i=0$ for each  $i\in [1,\lfloor d/2\rfloor-1]\cup[d+1,d+\lfloor d/2\rfloor-1]$.
Suppose to the contrary that $v_i\neq 0$ for $i\in [1,\lfloor d/2\rfloor-1]\cup[d+1,d+\lfloor d/2\rfloor-1]$.
Assume that $i\in [1,\lfloor d/2\rfloor-1]$. The case with $i\in [d+1,d+\lfloor d/2\rfloor-1]$ is symmetric. Since $\lceil d/2\rceil\geq d(\bv, \bc^2)\geq 1+d-|\supp(\bv)\cap \supp(\bc^2)|$, we have $w(\bv)\geq 1+ |\supp(\bv)\cap \supp(\bc^2)|\geq 1+\lceil d/2\rceil$. However, now $d(\bv,\bc_0)\geq 1+\lceil d/2\rceil$, a contradiction. Hence, $\supp(\bv)\subseteq [\lfloor d/2\rfloor,d]$. 

As $d(\bv, \bc^j) \leq \lceil d/2 \rceil$ for both $j\in\{1,2\}$, we have $\min\{|\vsupp(\bv)\cap\vsupp(\bc^1)|,|\vsupp(\bv)\cap\vsupp(\bc^2)|\}\geq \lfloor d/2\rfloor$. Thus, $v_i=1$ for each $i\in [\lfloor d/2\rfloor, d-3]$. 
Hence, since $w(\bv)\leq \lceil d/2\rceil$ and $d(\bv, \bc^j) \leq \lceil d/2 \rceil$ for both $j\in\{1,2\}$, we need $|\supp(\bv)\cap [d-2,d]|=2$. In particular, we need to have $v_i=1$ for exactly one $i\in[d-2,d]$ and the same is true for $v_j=2$ and $v_h=0$ for some $j,h\in[d-2,d]$. This gives exactly six possibilities for $\bv$ and \begin{equation}\label{eq|C|d+1}
\left|C\left(S^{\lceil d/2\rceil +1}_{d}\right)\right|=6\text.
\end{equation} That is, for $q<6$ there are more choices for $\bv$ with $A=\lceil d/2\rceil +1$ and for $q>6$ there are more choices for $\bv$ with $A=\lceil d/2\rceil$, giving the claim. 
%
%
%
%
%
%
\end{proof}

Note that interestingly, the previous theorem leaves the case with $q=6$ open. In particular, the technique which we use to prove the cases $q<6$ and $q>6$ does not work for $q=6$. 

\medskip

In the following theorem, we show that the triples presented in Theorem~\ref{thmMax3IntExact} are the ones which maximize the size of the intersection also when the distances between the centers are not equal. In its proof, we will utilize the following estimation for the distance between two words $\bx,\by\in \F^n_q$. Since $|\supp(\bx)\cap\supp(\by)|\geq|\vsupp(\bx)\cap\vsupp(\by)|$, by (\ref{eqDist}) we have \begin{equation}\label{eqDistEst}
w(\bx)+w(\by)-2|\supp(\bx)\cap\supp(\by)|\leq d(\bx,\by)\leq w(\bx)+w(\by)-2|\vsupp(\bx)\cap\vsupp(\by)|\text.
\end{equation}

\begin{theorem}\label{thmMax3IntExactNED}
Given fixed integers $q,t$ and $d$ such that $q\geq3$, $t\geq1$, $d\geq2$ and $n\geq n_0(q,t,d)$ for some sufficiently large threshold $n_0(q,t,d)$, the value $|I_t(S)|$ obtains its maximum over sets $S=\{\bc^0,\bc^1,\bc^2\} \subseteq \Z_q^n$ with $d_1=d(\bc^0,\bc^1)$, $d_2=d(\bc^0,\bc^2)$ and $d_3=d(\bc^1,\bc^2)$,  $d(S)\leq t$ and $\min\{d_1,d_2,d_3\}\geq d\geq2$, when $d_1=d_2=d_3=d$.
\end{theorem}
\begin{proof} Let us assume without loss of generality that $\bc^0=\zero$ and $d_3\geq d_2\geq d_1$.
By Lemma~\ref{lemMax3Int}, when $d_1=d_2=d_3=d$, we have $|I_t(S)|\in \Theta(n^{t-\lceil d/2\rceil})$. Moreover, by Lemma~\ref{lemAsymptoticDist}, we have $|I_t(S)|\in \Theta(n^{t-d(S)})$. 
Hence, if $d(S)\ge \lceil d/2 \rceil+1$, then 
the size $|I_t(S)|$ cannot attain the maximum cardinality. Observe that if $\lceil d_3/2\rceil\geq \lceil d/2\rceil+1$, then $d(S)\ge d(\{\bc^1,\bc^2\}) \geq \lceil d_3/2\rceil\geq \lceil d/2\rceil+1$ by applying Lemma~\ref{lemMindS}(1) to the set $\{\bc^1,\bc^2\}$. Thus, we may assume that $\lceil d_3/2\rceil\leq \lceil d/2\rceil$.
Furthermore, by Lemma~\ref{lemMindS}(1), we have $d(S)\geq \lceil d/2\rceil$. Thus, we may assume for the rest of the proof that $d(S)=\lceil d/2\rceil.$ 
Hence, for an even $d$ and a sufficiently large $n$, the maximum size intersection can be attained only by a set $S$ with $d_1=d_2=d_3$. Similarly, for an odd $d$ and the asymptotic maximum, it is enough to consider the following distances $$d\leq d_1\leq d_2\leq d_3\leq d+1\text.$$

In the following, we will split the proof into three cases for an odd $d$, first $d_1=d_2=d_3=d+1$, then $d_1+1=d_2=d_3=d+1$ and finally, $d_1+1=d_2+1=d_3=d+1$. At the end of the proof we will compare these cases with the case $d_1=d_2=d_3=d$. Recall that $d(S)=(d+1)/2$.  
In particular, this implies that for any central word $\bv$, we have $w(\bv)\leq (d+1)/2$. In each of the three cases, we will compute the number of central words.


Let us consider the case with $d_1=d_2=d_3=d+1$. Notice that we have $w(\bc^1)=w(\bc^2)=d+1$. In this case, due to the radius of the set, the value support of word $\bv$ intersects the value supports of $\bc^1$ and $\bc^2$ in at least $(d+1)/2$ locations. Hence, $w(\bv)=(d+1)/2$. Thus, we have $|\vsupp(\bc^1)\cap \vsupp(\bv)|=|\vsupp(\bc^2)\cap \vsupp(\bv)|=w(\bv)=(d+1)/2$. Thus, $|\vsupp(\bc^1)\cap \vsupp(\bc^2)|\geq (d+1)/2$. Since by Equation~(\ref{eqDistEst}) we have $d+1=d(\bc^1,\bc^2)\leq w(\bc^1)+w(\bc^2)-2|\vsupp(\bc^1)\cap \vsupp(\bc^2)|=2d+2-2|\vsupp(\bc^1)\cap \vsupp(\bc^2)|$, we have $|\vsupp(\bc^1)\cap \vsupp(\bc^2)|=(d+1)/2$. Hence, $\vsupp(\bc^1)\cap \vsupp(\bv)=\vsupp(\bc^2)\cap \vsupp(\bv)=\vsupp(\bc^1)\cap \vsupp(\bc^2)$ and therefore, we have only one possibility for $\bv$ and $|C(S)|=1$.

Let us next consider the case with $d_1+1=d_2=d_3=d+1$. We have $w(\bc^1)=d$ and $w(\bc^2)=d+1$. Hence, again due to the radius of the set $d(S)$, the value support of word $\bv$ intersects the value support of $\bc^2$ in at least $(d+1)/2$ locations. Thus, $w(\bv)=(d+1)/2$. Similarly, we may observe that $|\vsupp(\bv)\cap\vsupp(\bc^1)|\geq (d-1)/2$. Moreover,  by weight $w(\bv)$ and the lower bound in Equation~(\ref{eqDistEst}), we have $|\supp(\bv)\cap\supp(\bc^1)|= (d+1)/2$. Therefore, $|\vsupp(\bc^1)\cap\vsupp(\bc^2)|\geq (d-1)/2$ and $|\supp(\bc^1)\cap\supp(\bc^2)|\geq (d+1)/2$. Furthermore, since by Equation~(\ref{eqDist}) we have $d+1=d(\bc^1,\bc^2)=2d+1-|\vsupp(\bc^1)\cap \vsupp(\bc^2)|-|\supp(\bc^1)\cap \supp(\bc^2)|$, we have $|\supp(\bc^1)\cap \supp(\bc^2)|=(d+1)/2$ and $|\vsupp(\bc^1)\cap \vsupp(\bc^2)|=(d-1)/2$. Therefore, we have $\supp(\bv)=\supp(\bc^1)\cap\supp(\bc^2)$. Since we have $|\vsupp(\bv)\cap\vsupp(\bc^2)|=w(\bv)$, there is only one choice for $\bv$ and hence, $|C(S)|=1$ also in this case.

Finally, we are left with the case $d_1+1=d_2+1=d_3=d+1$. We have $w(\bc^1)=d$ and $w(\bc^2)=d$. Unlike in the previous cases, we have this time $(d-1)/2\leq w(\bv)\leq (d+1)/2$. Indeed, the value support of $\bv$ has to intersect both $\vsupp(\bc^1)$ and $\vsupp(\bc^2)$ in at least $(d-1)/2$ positions. Furthermore, as in the  cases above, we 
may compute that if $w(\bv)=(d-1)/2$, then 
$$(d-1)/2=w(\bv)\geq |\supp(\bv)\cap \supp(\bc^i)|\geq|\vsupp(\bv)\cap \vsupp(\bc^i)|\geq(d-1)/2$$ 
and thus, 
$|\supp(\bv)\cap \supp(\bc^i)|=|\vsupp(\bv)\cap \vsupp(\bc^i)|=(d-1)/2$. 
In other words, $\vsupp(\bv)\cap \vsupp(\bc^i)= \vsupp(\bv)$ and hence, in this case we have $|\vsupp(\bc^1)\cap \vsupp(\bc^2)|\geq (d-1)/2$.
Suppose that $|\vsupp(\bc^1)\cap \vsupp(\bc^2)|\geq (d+1)/2$. Then by Equation~\eqref{eqDistEst}, we have $d_3=d+1\leq 2d-2|\vsupp(\bc^1)\cap \vsupp(\bc^2)|\leq d-1$, a contradiction. Consequently, there can be at most one central word of weight $(d-1)/2$ and only when $|\vsupp(\bc^1)\cap \vsupp(\bc^2)|=(d-1)/2$.


Let us next count the number of central words $\bv$ of weight $w(\bv)=(d+1)/2$. Using Equation~\eqref{eqDist} for $d(\bv,\bc^i)$, we get  $|\supp(\bv)\cap \supp(\bc^i)|=(d+1)/2$ and 
$|\vsupp(\bv)\cap \vsupp(\bc^i)|\geq(d-1)/2$, for both $i\in\{1,2\}$. Therefore, we have   $|\vsupp(\bc^1)\cap \vsupp(\bc^2)|\geq(d-3)/2$ and $|\supp(\bc^1)\cap \supp(\bc^2)|\geq(d+1)/2$. Moreover, as $d(\bc^1,\bc^2)=d+1$, the equations above are tight due to Equation~\eqref{eqDist}. 
Hence, we have $|\vsupp(\bc^1)\cap \vsupp(\bc^2)|+2=|\supp(\bc^1)\cap \supp(\bc^2)|=(d+1)/2$. Since $w(\bv)= (d+1)/2$ and $|\supp(\bv)\cap \supp(\bc^i)|=(d+1)/2$, we have $\supp(\bv)= \supp(\bc^1)\cap \supp(\bc^2)$.
 We notice that since $|\vsupp(\bc^1)\cap \vsupp(\bc^2)|=(d-1)/2-1$ and $|\vsupp(\bv)\cap \vsupp(\bc^i)|\geq(d-1)/2$, we have $\vsupp(\bc^1)\cap \vsupp(\bc^2)\subseteq \vsupp(\bv)$. Except for the coordinate positions indicated by $\vsupp(\bv) \setminus (\vsupp(\bc^1)\cap \vsupp(\bc^2))$, all the symbols of $\bv$ are uniquely determined. Moreover, $|\vsupp(\bv) \setminus (\vsupp(\bc^1)\cap \vsupp(\bc^2))| = 2$ and within these two coordinates one symbol has to agree with $\bc^1$ and the other symbol with $\bc^2$. Therefore, we have exactly two choices for $\bv$ and $|C(S)|=2$. Therefore, we have $|C(S)|=2$ when  $|\vsupp(\bc^1)\cap \vsupp(\bc^2)|=(d-1)/2-1$, and 
$|C(S)|=1$ when  $|\vsupp(\bc^1)\cap \vsupp(\bc^2)|=(d-1)/2$.


Overall, in each case we have shown that $|C(S)|\leq 2$. However, in Equations~(\ref{eq|C|d}) and~(\ref{eq|C|d+1}) for an odd $d$, we have seen that for $d_1=d_2=d_3$ we have $|C(S)|\in\{6,q\}$ where $q\geq3$. Hence, by Lemma~\ref{lemAsymptoticDistCent}, the maximum intersection is obtained by the case $d_1=d_2=d_3=d$, as claimed.
\end{proof}

\subsubsection{Maximum intersection size \texorpdfstring{$N^q_t(3,d)$}{Nqt(3,d)}} \label{SubsubSecMaxIntWithFormula}

In this section, we focus on determining the exact size of $N^q_t(3,d)$ with the aid of Theorem~\ref{Thm_intersection_size} for sufficiently large $n$. Recall that the case with $q=2$ is known by~\cite{yaakobi2018uncertainty}. Thus, we assume $q\geq3$. Let $S=\{\bc^1,\bc^2,\bc^3\}\subseteq\Z^n_q$ with $\dmin(S)=d$, where elements of $S$ are named as in Theorem~\ref{Thm_intersection_size}. By Theorem~\ref{thmMax3IntExactNED}, in order to determine $N^q_t(3,d)$, we may assume that $d(\bc^1,\bc^2) = d(\bc^1,\bc^3) = d(\bc^2,\bc^3) = d \ (\geq 2)$. Furthermore, we assume without loss of generality that $\bc^1=0^n$. Let us first consider the case with an \emph{even} $d$.  Based on Theorem~\ref{thmMax3IntExact}, for $|I_t(S)|$ to be equal to $N^q_t(3,d)$, it follows that $A = d/2$. Therefore, as $A_2=2A-d=0$, we may further notice that there are no coordinates in which all the words $\bc^1$, $\bc^2$ and $\bc^3$ disagree, i.e., using the notation of Theorem~\ref{Thm_intersection_size} $p_{3,1} = 0$. 
Hence, we may without loss of generality assume that 
$\bc^2=1^d0^{n-d}$ and  $\bc^3=0^{d/2}1^d0^{n-3d/2}$. For the parameters $p_{k,i}$ (other than $p_{3,1}$) of Theorem~\ref{Thm_intersection_size}, it follows that $p_{1,1}=n-\frac{3d}{2}$ and $p_{2,1}=p_{2,2}=p_{2,3}=\frac{d}{2}$. In what follows, we show that for even $d$, $q \geq 3$ and sufficiently large $n$, 
\begin{equation}\label{eqNqt3d} 
    N^q_t(3,d)=\sum_{\mathcal{J}} \left(   \binom{n-\frac{3d}{2}}{j_{1,1}^1}(q-1)^{j_{1,1}^{2}}\prod_{i=1}^{3} \binom{\frac{d}{2}}{j_{2,i}^1,j_{2,i}^2,j_{2,i}^3}(q-2)^{j_{2,i}^{3}} \right)  \text,
\end{equation}
where the summation goes through $\mathcal{J}$ which can be efficiently computed due to Remark~\ref{RmkComplexity} and  consists of all the tuples $j_{k,i}^l$ ($k \in \{1, 2, 3\}$, $i \in \{1, 2, \dots, \cS(3,k)\}$ and $l \in \{1, 2, \dots, \min\{q,k+1\}\}$) satisfying the following conditions: (i) $\sum_{l=1}^{\min\{q,k+1\}} j_{k,i}^l = p_{k,i}$ and (ii) for each $\bc^m \in S$, $\sum_{k=1}^{2} \sum_{i=1}^{\cS(3,k)} \left( p_{k,i} - j_{k,i}^{f_{k,i}(\bc^m)} \right) \leq t$. 
Indeed, Equation~\eqref{eqNqt3d} follows directly from Remark~\ref{remSimplifiedFormula} as follows:
\begin{align*}
    N^q_t(3,d) &=|I_t(S)|=\sum_{\mathcal{J}} \left( \prod_{k=1}^{q-1}\left(\prod_{i=1}^{\mathcal{S}(3,k)}\binom{p_{k,i}}{j_{k,i}^1, j_{k,i}^2, \dots, j_{k,i}^{k+1}}(q-k)^{j_{k,i}^{k+1}} \right)\cdot\prod_{i=1}^{\mathcal{S}(3,q)}\binom{p_{q,i}}{j_{q,i}^1, j_{q,i}^2, \dots, j_{q,i}^{q}} \right)  \\
    &=\sum_{\mathcal{J}} \left( \prod_{k=1}^{2}\left(\prod_{i=1}^{\mathcal{S}(3,k)}\binom{p_{k,i}}{j_{k,i}^1, j_{k,i}^2, \dots, j_{k,i}^{k+1}}(q-k)^{j_{k,i}^{k+1}} \right)\cdot\prod_{i=1}^{\mathcal{S}(3,q)}\binom{0}{j_{q,i}^1, j_{q,i}^2, \dots, j_{q,i}^{q}} \right)\\
    &=\sum_{\mathcal{J}} \left(   \binom{n-\frac{3d}{2}}{j_{1,1}^1}(q-1)^{j_{1,1}^{2}}\prod_{i=1}^{3} \binom{\frac{d}{2}}{j_{2,i}^1,j_{2,i}^2,j_{2,i}^3}(q-2)^{j_{2,i}^{3}} \right)\text, 
\end{align*}
	    where the sum goes through all the tuples of $\mathcal{J}$. The third equality follows from $p_{3,i}=0$ and $\mathcal{S}(3,k)=0$ for $k>3$. The fourth inequality follows from $\mathcal{S}(3,1)=1$ and $\mathcal{S}(3,2)=3$.

\smallskip

Let us now consider the case $d$ is \emph{odd}. 
Recall that if $q< 6$, the maximum $|I_T(S)|$, that is, $N_t^q(3,d)$ is obtained by Theorem~\ref{thmMax3IntExact} for $A = (d+1)/2+1 = (d+3)/2$ further implying that $A_1 = d-A = (d-3)/2$ and $A_2=2A-d = 3$. Therefore, similarly to the case with even $d$, we obtain that $p_{1,1}=n-\frac{3d-3}{2}$, $p_{2,1}=p_{2,2}=p_{2,3}=\frac{d-3}{2}$, and $p_{3,1}=3$. 
By substituting these values into Theorem~\ref{Thm_intersection_size}, we divide into the following cases depending on $q \in \{3, 4, 5\}$:
\begin{itemize}
    \item If $q=3$, then $N_t^q(3,d)$ (for sufficiently large $n$) is equal to
    \begin{equation}
        \sum_{\mathcal{J}} \left( \binom{n-\frac{3d-3}{2}}{j_{1,1}^1}2^{j_{1,1}^{2}} \prod_{i=1}^{3} \binom{\frac{d-3}{2}}{j_{2,i}^1,j_{2,i}^2,j_{2,i}^3} \cdot \binom{3}{j_{3,1}^1,j_{3,1}^2,j_{3,1}^3}\right) \text,   
\end{equation}
    where the sum goes through all the tuples of $\mathcal{J}$.

    \item If $4 \leq q \leq 5$, then $N_t^q(3,d)$  (for sufficiently large $n$) is equal to
    \begin{equation}
    \sum_{\mathcal{J}} \left( \binom{n-\frac{3d-3}{2}}{j_{1,1}^1}(q-1)^{j_{1,1}^{2}} \prod_{i=1}^{3} \binom{\frac{d-3}{2}}{j_{2,i}^1,j_{2,i}^2,j_{2,i}^3}(q-2)^{j_{2,i}^{3}} \cdot \binom{3}{j_{3,1}^1,j_{3,1}^2,j_{3,1}^3,j_{3,1}^4}(q-3)^{j_{3,1}^{4}}\right) \text,   
    \end{equation}
    where the sum goes through all the tuples of $\mathcal{J}$.
\end{itemize}

Finally, assume that $q \geq 7$ (and $d$ is odd). Now, based on Theorem~\ref{thmMax3IntExact}, we require that $A = (d+1)/2$. This implies that $A_1 = d-A = (d-1)/2$ and $A_2=2A-d = 1$. Therefore, as in the previous cases, we obtain $p_{1,1}=n-\frac{3d-1}{2}$, $p_{2,1}=\frac{d-1}{2}$, $p_{2,2}=\frac{d-1}{2}$, $p_{2,3}=\frac{d-1}{2}$, and $p_{3,1}=1$.
In this case, we have  (for sufficiently large $n$) by Theorem~\ref{Thm_intersection_size}
\begin{equation}
N_t^q(3,d) = \sum_{\mathcal{J}} \left( \binom{n-\frac{3d-1}{2}}{j_{1,1}^1}(q-1)^{j_{1,1}^{2}}(q-3)^{j_{3,1}^{4}}\prod_{i=1}^{3} \binom{\frac{d-1}{2}}{j_{2,i}^1,j_{2,i}^2,j_{2,i}^3}(q-2)^{j_{2,i}^{3}}\right)\text,
\end{equation}
where the sum goes through all the tuples of $\mathcal{J}$.

We remark that for an odd $d$ and $q=6$, the value $N_t^6(3,d)$ remains unknown, as the case with an odd $d$ and $q=6$ is not solved in Theorem~\ref{thmMax3IntExact}.


\section{Conclusion}\label{secConclusion}
We studied the cardinality of the intersection of $s\geq2$ substitution balls of varying radii in $q$-ary Hamming spaces. First, in Section~\ref{secGeneralFormula}, we introduced an exact formula for computing the size of the intersection with the help of the Stirling numbers of the second kind. Our formula works for any $n$, $q\geq2$, $s\geq2$ and set $T$ of radii. Based on the exact formula, we continued in Section~\ref{subsecNumPara} by discussing about the additional information required for computing the size of the intersection besides the pairwise distances between the centers. In particular, we observed that if $s\geq4$, or $s\geq3$ and $q\geq3$, some additional information is necessarily required.

Continuing in Section~\ref{SecIntersection}, we discussed the asymptotic value of the intersection with the help of the radius $d(S)$ of the set $S$ and the central words $C(S)$. Then, in Section~\ref{sec3balls}, we considered the special case of exactly three Hamming balls. 
In Subsection~\ref{SubsubSecMaxInt}, we characterize the configurations of centers with pairwise distances at least $d$ for which the intersection of the balls is maximal, for all $q\ge 3$ with $q\neq 6$.
In Subsection~\ref{SubsubSecMaxIntWithFormula}, we combined these results with Theorem~\ref{Thm_intersection_size} to obtain exact formulas for the (previously unknown) values $N^q_t(3,d)$ when $q\geq 3$, $q\neq 6$ and $n$ is large enough. These values, especially for $q=4$, are important in the study of Levenshtein's sequence reconstruction problem due to its connections to DNA data storage systems, as discussed in the introduction.

For future research, it would be interesting to find a complete solution for the size of the maximum intersection when $q=6$. We believe that such solution might be possible to obtain by studying the number of the words $\bv\in\F^n_q$ with $d_M(\bv,S)=d(S)+1$ in both cases $A=(d+1)/2$ and $A=(d+1)/2+1$ of Theorem~\ref{thmMax3IntExactNED}.
Another interesting question is whether there exist some universal conditions 
for the centers of the Hamming balls which obtain the maximum size intersection when we consider more than three balls (instead of exactly three). With the help of our Theorem~\ref{Thm_intersection_size}, such a description could lead to a general formula for the value $N^q_t(s,d)$, which would be of interest as it is related to DNA-storage as mentioned in the introduction. Further lines of research include describing the algorithmic complexity of the problem, in particular when the number of balls $s$ is not bounded by a constant. While such cases are covered by Theorem~\ref{Thm_intersection_size}, the complexity of any related algorithm is unclear, as we have to solve a system of a non-constant number of linear inequalities. 

\section*{Acknowledgment}
We thank the referees for their insightful comments and corrections, which helped to improve the paper.

\bibliographystyle{IEEEtran}
\bibliography{ITW25}

\end{document}